\documentclass[11pt]{amsart}

\usepackage{latexsym}
\usepackage{amssymb}
\usepackage{amsmath}
\usepackage{color}
\usepackage{tikz-cd}

\newtheorem{theorem}{Theorem}[section]
\newtheorem{lemma}[theorem]{Lemma}
\newtheorem{proposition}[theorem]{Proposition}
\newtheorem{corollary}[theorem]{Corollary}
\newtheorem{definition}[theorem]{Definition}

\newtheorem{remark}[theorem]{Remark}

\newcommand\id{\mathop{\rm id}}
\newcommand\tr{\mathop{\rm tr}}

\newcommand{\cl}[1]{\mathcal{#1}}

\newcommand{\ignore}[1]{ { } }

\usepackage{fancyhdr}
\begin{document}
\author{
  Alexey Kuzmin,
  Lyudmila Turowska
  \\
}
\title{Classification of irrational $\Theta$-deformed CAR $C^*$-algebras}
\maketitle
\begin{abstract}
Given a skew-symmetric real $n\times n$ matrix $\Theta$ we consider the universal enveloping $C^*$-algebra $\mathsf{CAR}_\Theta$ of the  $*$-algebra generated by $a_1, \ldots, a_n$ subject to the relations
\[ a_i^* a_i + a_i a_i^* = 1, \ \]
\[ a_i^* a_j = e^{2 \pi i \Theta_{i,j}} a_j a_i^*, \]
\[ a_i a_j = e^{-2 \pi i \Theta_{i,j}} a_j a_i. \]
We prove that $\mathsf{CAR}_\Theta$ has a $C(K_n)$-structure, where $K_n = \left[ 0,\frac{1}{2} \right]^n$ is the hypercube and describe the fibers. We classify irreducible  representations of $\mathsf{CAR}_\Theta$ in terms of irreducible representations of a higher-dimensional noncommutative torus. We prove that for a given irrational skew-symmetric $\Theta_1$ there are only finitely many $\Theta_2$ such that $\mathsf{CAR}_{\Theta_1} \simeq \mathsf{CAR}_{\Theta_2}$. Namely, $\mathsf{CAR}_{\Theta_1} \simeq \mathsf{CAR}_{\Theta_2}$ implies $(\Theta_1)_{ij} = \pm (\Theta_2)_{\sigma(i,j)} \mod \mathbb{Z}$ for a bijection $\sigma$ of the set $\{(i,j) : i < j, \ i, j = 1, \ldots, n\}$. For $n = 2$ we give a full classification: $\mathsf{CAR}_{\theta_1} \simeq \mathsf{CAR}_{\theta_2}$ iff $\theta_1 = \pm \theta_2 \mod \mathbb{Z}$.
\end{abstract}
\tableofcontents

\section{Introduction}
One of the most well-studied examples of noncommutative manifolds are the noncommutative tori, see \cite{rieff}. Given a real skew-symmetric $n\times n$ matrix $\Theta=(\Theta_{i,j})$ the noncommutative torus $C(\mathbb{T}^n_\Theta)$ is defined as the universal $C^*$-algebra generated by $n$ unitaries $u_1, \ldots, u_n$ subject to the relations
\[ u_i u_j = e^{-2 \pi i \Theta_{i,j}} u_j u_i. \]

The problem of classification of $C(\mathbb{T}^n_\Theta)$ up to $C^*$-isomorphism has been solved  in \cite{phillips}  in the case when $\Theta$ is {\it irrational}. 
In particular, in the case $n = 2$, identifying $\Theta$  with $\Theta_{1,2} = \theta$ we have $C(\mathbb{T}^2_{\theta_1}) \simeq C(\mathbb{T}^2_{\theta_2})$ iff $\theta_1 = \pm \theta_2 \mod \mathbb{Z}$. For rational $\Theta$ the classification is given in \cite{brenken}.

In this paper we study the universal enveloping $C^*$-algebra $\mathsf{CAR}_\Theta$ of the $*$-algebra generated by $a_1, \ldots, a_n$ subject to the relations
\[ a_i^* a_i + a_i a_i^* = 1, \ \]
\[ a_i^* a_j = e^{2 \pi i \Theta_{i,j}} a_j a_i^*, \]
\[ a_i a_j = e^{-2 \pi i \Theta_{i,j}} a_j a_i. \]
The representation theory of $\mathsf{CAR}_\Theta$ was studied in  \cite{proskurin-sukretnyi} and it appeared to be related to representation theory of a noncommutative torus. In this paper we in particular explain and reprove the result by showing that $\mathsf{CAR}_\Theta$ has a $C(K_n)$-structure for $K_n = \left[0,\frac{1}{2}\right]^n$ with fibers being isomorphic to matrix algebras over crossed products of noncommutative tori by  finite groups. The description of $\mathsf{CAR}_\Theta$ as a "noncommutative fiber bundle" allows us to establish a result about classification  $\mathsf{CAR}_\Theta$ up to isomorphism for irrational $\Theta$ which was the main motivation to pursue our study of the object. 

 Noncommutative tori have been playing a role of a training ground for testing various ideas in noncommutative geometry and topology. Such questions as classification up to $C^*$-isomorphism, classification of projective modules, classification up to Morita equivalence, construction of Dirac operators, study of quantum metric structures, construction of pseudodifferential calculi, study of index theory, generalizations of the notion of curvature and many other have been studied for $C(\mathbb{T}^n_\Theta)$. Because of the simplicity of the algebraic definition of $\mathsf{CAR}_\Theta$ and the existence of a noncommutative fiber bundle structure on $\mathsf{CAR}_\Theta$ with the fibers resembling noncommutative tori, it is natural to ask the same questions about its structure as for $C(\mathbb{T}^n_\Theta)$. In this paper we are interested in the noncommutative topology of $\mathsf{CAR}_\Theta$, in particular the classification of $\mathsf{CAR}_\Theta$.
We prove that $\mathsf{CAR}_{\theta_1} \simeq \mathsf{CAR}_{\theta_2}$ for irrational $\theta_1, \theta_2$ and $n = 2$ iff $\theta_1 = \pm \theta_2 \mod \mathbb{Z} $. Moreover for general $n$ and irrational $\Theta_1$, $\Theta_2$ we prove that $\mathsf{CAR}_{\Theta_1} \simeq \mathsf{CAR}_{\Theta_2}$ implies that  $(\Theta_1)_{i,j} = \pm (\Theta_2)_{\sigma(i,j)} \mod \mathbb{Z}$ for a bijection $\sigma$ of the set $\{(i,j) : i < j, \ i, j = 1, \ldots, n\}$. 

The general idea for our analysis of $\mathsf{CAR}_\Theta$ is to express it as Rieffel's deformation of $n$ tensor copies of $\mathsf{CAR}_1$ - the one-dimensional $\mathsf{CAR}$-algebra, structure of which is well-understood: it is a $C(\left[0,\frac{1}{2}\right])$-$C^*$-algebra with well-known fibers. Then we use the fact that Rieffel's deformation of a $C_0(X)$-$C^*$-algebra also has a $C_0(X)$-structure with fibers which are Rieffel's deformations of the fibers of the undeformed $C^*$-algebra. 

The structure of the article is the following: in Section 2 and 3 we recall some relevant facts from the theory of $C_0(X)$-$C^*$-algebras and Rieffel's deformations. In Section 4 we prove an isomorphism between Rieffel's deformation of the matrix algebra over a $C^*$-algebra $A$ and the matrix algebra over Rieffel's deformation of $A$. Although this result has been known in the literature, here we construct an explicit isomorphism, which will be used in further sections. In Section 5 we show that the $C^*$-algebra $\mathsf{CAR}_\Theta$ is isomorphic to Rieffel's deformation of $\mathsf{CAR}_1^{\otimes n}$. In Section 6 we give an analysis of $\mathsf{CAR}_1$ - we describe its representation theory, show that it has a $C(\left[0,\frac{1}{2}\right])$-structure and describe fibers with respect to this structure. In Section 7 we transfer the described structural features of $\mathsf{CAR}_1$ first to $\mathsf{CAR}_1^{\otimes n}$ and then to its Rieffel deformation, that allows to  obtain an alternative proof for classification of irreducible representations of $\mathsf{CAR}_\Theta$ (Theorem \ref{clas_irr}). In Section 8 we further exploit the noncommutative fiber bundle structure of $\mathsf{CAR}_\Theta$ and prove the classification result (Theorem \ref{clas_tori}, Corollary \ref{clas_cor}).

We believe that the $C^*$-algebra $\mathsf{CAR}_{\Theta}$ is a nice rich object to study other questions of noncommutative geometry and this will be pursued elsewhere. 

\section{$C_0(X)$-structure on $C^*$-algebras}
Let $X$ be a locally compact Hausdorff space and let $C_0(X)$ be the $C^*$-algebra of continuous functions on $X$ that vanish at infinity.  For a $C^*$-algebra $A$ write $\mathcal M(A)$ to denote its multiplier algebra; let $Z(A)$ be its center. $C_0(X,A)$ will stand for the algebra of $A$-valued continuous functions on $X$ that vanish at infinity. 

\begin{definition}
A $C_0(X)$-structure on a $C^*$-algebra $A$ is a monomorphism 
\[ \Phi : C_0(X) \rightarrow Z(\mathcal{M}(A)) \]
such that the ideal $\Phi(C_0(X)) \cdot A$ is dense in $A$. In this case we say that $A$ is a  $C_0(X)$-$C^*$-algebra.
\end{definition}

Let $A$ be a $C_0(X)$-$C^*$-algebra. For $x \in X$ consider  the closed ideal 
\[ I_x^\Phi = \overline{\text{span}\{ \Phi(f) \cdot a, \ a \in A, \ f \in C_0(X) \text{ such that } f(x) = 0 \}}. \]
The fiber $A^\Phi(x)$ of $A$ over $x$ is defined as
\[ A^\Phi(x) = A / I_x^\Phi, \]  
and the canonical quotient map ${\rm ev}_x^\Phi : A \rightarrow A^\Phi(x)$ will be  called the evaluation map at $x$. When $C_0(X)$-structure $\Phi$ is evident from the context, we will simply write $I_x$, $A(x)$ and ${\rm ev}_x$; we shall often write $a(x)$ instead of ${\rm ev}_x(a)$; it is also common to suppress mention of $\Phi$ and simply write $f\cdot a$ instead of $\Phi(f)\cdot a$. 

Let $G$ be a locally compact group and $\alpha:G\to \text{Aut}(A)$ be a continuous (with respect to point norm topology) group homomorphism, which we call an action of $G$ on $A$; thus $(A,G,\alpha)$ is a $C^*$-dynamical system. 
\begin{definition}
Let $A$ be a $C_0(X)$-$C^*$-algebra. We say that $\alpha$  is  fiberwise if
\[ \alpha_g( f\cdot a) = f\cdot(\alpha_g(a)), \ g \in G, \ a \in A, \ f \in C_0(X). \]
\end{definition}
If an action  $\alpha$ is fiberwise then it induces an action $\alpha^x$ of $G$ on $A(x)$ for every $x \in X$ by letting $\alpha_g^x(a(x))=\alpha_g(a)(x)$, $a\in A$, $g\in G$.

\section{Rieffel deformation}\label{Rieffeldef}
We turn now to Rieffel's deformation \cite{rieffel_primar}, that will be essential for our consideration, and recall main constructions needed for the paper. 

Given a $C^*$-dynamical system $(A,\mathbb R^n,\alpha)$, let $A^\infty$ denote the set of all $a \in A$ such that $t \mapsto \alpha_t(a)$ is a $C^\infty$-function. It is a dense $*$-subalgebra of $A$. 
Let $\Theta$ be a real skew-symmetric $n\times n$-matrix. To define Rieffel's deformation, one keeps the involution unchanged and introduces on $A^\infty$ the product defined by oscillatory integrals
\begin{equation}\label{rieffel_product}
    a \cdot_\Theta b := 
\int_{\mathbb{R}^n} \int_{\mathbb{R}^n} \alpha_{\Theta(x)}(a) \alpha_y(b) e^{2 \pi i \langle x, y \rangle} dx dy,
\end{equation}
where $\langle x, y \rangle$ is the inner product on $\mathbb{R}^n$. The $*$-algebra $(A^\infty, \cdot_\Theta)$ admits a $C^*$-completion $A^\Theta$ in a $C^*$-norm, defined by Hilbert module techniques. The action $\alpha$ leaves $A^\infty$ invariant and extends to the action $\alpha^\Theta$ on the $C^*$-algebra $A^\Theta$. 
 More generally, any equivariant $*$-homomorphism $f$ between $C^*$-algebras $A$ and $B$ with actions $\alpha^A$ and $\alpha^B$ of $\mathbb{R}^n$ respectively (i.e. $f(\alpha_x^A(a))=\alpha_x^B(f(a))$, $a\in A$, $x\in\mathbb R^n$) can be lifted to a $*$-homomorphism $f^\Theta : A^\Theta \rightarrow B^\Theta$, which is also equivariant. We refer the reader to \cite{rieffel_primar} for these and other details concerning the construction. Through this section we will keep notation $\alpha$ for the action that defines Rieffel's deformation.

\medskip{}

The procedure of Rieffel's deformation is invertible; the next statement follows from \cite{Kasprzak_rieffel}, Lemma 3.5. 
\begin{proposition}\label{double_deformation}
The identity mapping  extends to a $*$-isomorphism $\text{id} : A \rightarrow (A^\Theta)^{-\Theta}$.
\end{proposition}
In nice  situations Rieffel's deformation of a $C_0(X)$-algebra is also a $C_0(X)$-algebra:
\begin{proposition}[\cite{Belmonte}, Proposition 4.4]\label{cx_deformation}
Let $\alpha$ be a fiberwise action of $\mathbb{R}^n$ on a $C_0(X)$-$C^*$-algebra $A$. Then Rieffel's deformation $A^\Theta$ possesses a $C_0(X)$-structure such that
\[ (A^\Theta)(x) \simeq (A(x))^\Theta, x\in X. \]
\end{proposition}

We will need to know how crossed product $C^*$-algebras are transformed under Rieffel's deformation. 

Given a $C^*$-dynamical system $(A,G,\sigma)$, write $A\rtimes_\sigma G$ for the corresponding full or reduced crossed product $C^*$-algebra (\cite{williams}) and denote by $A^\sigma$ the set of fixed points of $A$, i.e.
\[ A^\sigma = \{ a \in A : \sigma_g(a) = a \text{ for every } g \in G\}. \]
If $\alpha$ is an action of $\mathbb R^n$ on $A$ such that 
\begin{equation}\label{actions}
 \sigma_g (\alpha_t(a)) = \alpha_t(\sigma_g(a)), \text{ for all } g \in G, t \in \mathbb{R}^n, a \in A, \end{equation} then $\alpha$ extends to an action on $A\rtimes_\sigma G$ by letting
 $$\alpha_t(f)(g)=\alpha_t(f(g)), f\in C_c(G,A).$$
 The next proposition identifies Rieffel's deformations of $A^\sigma$ and $A\rtimes_\sigma G$. 
\begin{proposition}\label{fixed_point_deformation}
Let $(A,G,\sigma)$ be a $C^*$-dynamical system and let $\alpha$ be an $\mathbb R^n$-action on $A$ which satisfies (\ref{actions}) and hence extends to the $\mathbb R^n$-action on $A\rtimes_\sigma G$ as above. Let 
$\Theta$ be a real skew-symmetric $n\times n$ matrix. 
Then $(a,g) \in (A^\Theta, G) \mapsto (\sigma_g)^\Theta(a) \in A^\Theta$ defines an action $\sigma^\Theta$ of $G$ on $A^\Theta$ such that 
\[ (A^\sigma)^\Theta \simeq (A^\Theta)^{\sigma^\Theta} \]
and
\[ (A \rtimes_\sigma G)^\Theta \simeq (A^\Theta) \rtimes_{\sigma^\Theta} G. \]
\end{proposition}
\begin{proof}
The identity map $(A^\sigma)^\Theta \rightarrow (A^\Theta)^{\sigma^\Theta}$ and $(A \rtimes_\sigma G)^\Theta \rightarrow (A^\Theta) \rtimes_{\sigma^\Theta}G$ gives the isomorphism. The only nontrivial thing is to show the homo\-mor\-phism property in the second case.  Let $f, g \in C_c(G)\odot A^\infty$, the algebraic tensor product of $C_c(G)$ and $A^\infty$.  One has that $f$, $g$ are smooth elements of $A \rtimes_\sigma G$, and writing the deformed product as convolution $\ast_\Theta$ we obtain
\begin{align*}
    (f \ast_\Theta g)(s) & = \int_{\mathbb{R}^n} \int_{\mathbb{R}^n} \int_G \alpha_{\Theta(x)}(f)(t) \sigma_t(\alpha_y(g)(t^{-1}s)) e^{2\pi i \langle x, y \rangle} dt dx dy = \\ & = \int_G \int_{\mathbb{R}^n} \int_{\mathbb{R}^n} \alpha_{\Theta(x)}(f(t)) \alpha_y(\sigma_t(g(t^{-1}s))) e^{2\pi i \langle x, y \rangle} dx dy dt = \\ & = \int_G f(t) \cdot_\Theta \sigma_t(g(t^{-1}s)) dt, 
\end{align*}
where the latter is the convolution determined by  $(A^\Theta, \sigma^\Theta, G)$. 
\end{proof}

In this paper we will be interested in periodic actions of $\mathbb{R}^n$, i.e. we assume that $\alpha$ is an action of $\mathbb{T}^n$. 
Given a character $\chi \in \widehat{\mathbb{T}}^n \simeq \mathbb{Z}^n$, consider the associated spectral subspace
\[ 
A_\chi = \{ a \in A : \alpha_z(a) = \chi(z)a \text{ for every } z \in \mathbb{T}^n \} .
\]
Then 
\[
A=\overline{\text{span}\bigcup_{\chi\in\mathbb Z^n} A_{\chi}}
\]
and $A_{\chi_1}\cdot A_{\chi_2}\subset A_{\chi_1+\chi_2}$, $A_{\chi}^*=A_{-\chi}$; hence $A_\chi$, $\chi \in \mathbb Z^n$, can be treated as homogeneous components of the induced
$\mathbb Z^n$-grading on $A$. 

For $p=(p_1,\ldots, p_n)\in \mathbb{Z}^n\simeq  \mathbb{\widehat{T}}^n$,  we will write $\chi_p$ for the character of $\mathbb T^n$ given by $\chi_p(z)=z_1^{p_1}\ldots z_n^{p_n}$, $z=(z_1,\ldots,z_n)$, and write $A_p$ instead of $A_{\chi_p}$. 

\medskip{}

For the action of $\mathbb T^n$, one has an explicit formula for the deformed product of homogeneous elements.
\begin{proposition}[\cite{rieffel_primar}, Proposition 2.22]\label{homogeneous_rieffel}
Suppose $A$ is a $C^*$-algebra with a $\mathbb{T}^n$-action. Assume that $a \in A_p$, $b \in A_q$  for $p$, $q\in\mathbb Z^n$. Then
\[ a \cdot_\Theta b = e^{2 \pi i \langle \Theta(p), q \rangle}a \cdot b . \]
\end{proposition}
\noindent
Consider a $C^*$-dynamical system $(A,\mathbb T^n, \alpha)$ and its covariant 
representation $(\pi,U)$ on a Hilbert space $\mathcal H$, i.e.  $\pi(\alpha_z(a))=U_z\pi(a)U_z^*$, $a\in A$, $z\in\mathbb T^n$.  
For $p\in \mathbb Z^n$ consider the spectral space
\[
 \mathcal H_p = \{ h\in \mathcal H \mid U_z h = \chi_p(z) h \text{ for all }z\in\mathbb T^n\}.
\]
Then $\mathcal H = \bigoplus _{p\in \mathbb Z^n} \mathcal H_p$ (see \cite{williams}).

The next result describes a procedure how to lift the representation $\pi$ of $A$ to a representation of its Rieffel deformation. 
\begin{proposition}[\cite{warped}, Theorem 2.8]\label{theta_rep}
Let $(\pi,U)$ be a covariant representation of $(A,\mathbb T^n,\alpha)$ on a Hilbert space $\mathcal{H}$. Then $\pi^\Theta$,  given by 
\[ 
\pi^\Theta(a) \xi = e^{2 \pi i \langle \Theta(p), q \rangle} \pi(a) \xi,  
\]
for $\xi \in \mathcal{H}_q$, $a \in A_p$, $p,q\in \mathbb Z^n$, extends to a $*$-representation of $A^\Theta$. 
Moreover, $\pi^\Theta$ is faithful if and only if $\pi$ is faithful. 
\end{proposition}

\begin{remark}\label{trivial_action_remark}
Notice that if the action of $\mathbb{R}^n$ on $A\otimes B$ is given by $\alpha = id \otimes \alpha_B$, where $\alpha_B$ is an $\mathbb R^n$ action on $B$, then 
\[ (A \otimes B)^\Theta \simeq A \otimes B^\Theta. \]
\end{remark}

We have also the following invariance of $K$-groups under Rieffel's deformation. 
\begin{proposition}[\cite{Kasprzak_rieffel}, Theorem 3.13]\label{Rieff_K_theory}
For a $C^*$-algebra $\mathcal A$ one has
\[
K_0(\mathcal A^{\Theta}) = K_0(\mathcal A)\quad \mbox{and}\quad K_1(\mathcal A^{\Theta})=K_1(\mathcal A).
\]
\end{proposition}

\section{Rieffel deformation of $M_n(A)$}\label{riefell_mna}
In the sequel we will need to work with  Rieffel's deformation of the matrix algebra over a $C^*$-algebra $A$, which we will describe in this section.  

Suppose $\mathbb{T}^k$ acts on $\mathbb{C}^n$ by unitaries $U_z$, $z\in \mathbb T^k$, i.e. $z\mapsto U_z$ is a strongly continuous representation of $\mathbb T^k$ on $\mathbb C^n$. It induces an  action of $\mathbb{T}^k$ on $M_n$ given by $\alpha_z(X)\xi = U_zX U_z^*\xi$, $z\in \mathbb T^k$, $X\in M_n$, $\xi\in \mathbb C^n$; thus $(M_n,\mathbb T^k,\alpha)$ is a $C^*$-dynamical system and $(\text{id}, U)$ is its covariant  representation on $\mathbb C^n$, where $\text{id}$ is the identity representation of $M_n$ when the latter is identified with  $B(\mathbb C^n)$. These actions define $\mathbb{Z}^k$-gradings on $\mathbb{C}^n$ and $M_n$ as in Section \ref{Rieffeldef}. 
The following lemma is a direct consequence of Proposition \ref{theta_rep}.
\begin{lemma}\label{lemma51}
Let $\Theta$ be a real skew-symmetric $k\times k$ matrix and let $\Psi : M_n^\Theta \rightarrow M_n$ be given by  
\[ \Psi(a)\xi = e^{2 \pi i \langle \Theta(p), q \rangle} a \xi,  \]
where $a \in M_n$ is homogeneous of order $p \in \mathbb{Z}^k$ and $\xi \in \mathbb{C}^n$ is homogeneous of order $q \in \mathbb{Z}^k$. Then $\Psi$
is an equivariant $*$-isomorphism from $(M_n^\Theta, \mathbb T^k,\alpha^\Theta)$ to $(M_n,\mathbb T^k,\alpha)$.
\end{lemma}

Let $(A,\mathbb T^m,\alpha^A)$ be a $C^*$-dynamical system and consider  the  action of $\mathbb{T}^{k+m}$ on $M_n \otimes A$ given by $X\otimes a\mapsto\alpha_{z_1}(X)\otimes\alpha_{z_2}^A(a)$, $(z_1,z_2)\in\mathbb T^k\times\mathbb T^m$, $X\in M_n$, $a\in A$.

Let $\Theta$ be a real skew-symmetric matrix of size $k+m$ and consider its block partition
$\Theta=\left(\begin{array}{cc}\Theta_{1,1}&\Theta_{1,2}\\\Theta_{2,1}&\Theta_{2,2}\end{array}\right)$, where $\Theta_{1,1}\in M_k$ and $\Theta_{2,2}\in M_m$. 

\smallskip

For $p \in \mathbb{Z}^m$ set $\omega_l(p) = e^{2\pi i \langle \Theta_{2,1}(\epsilon_l),p\rangle} \in \mathbb{T}$ , $l=1,\ldots, k$, and  $\omega(p)=(\omega_1(p),\ldots,\omega_k(p))
\in \mathbb T^k$; here $\{\epsilon_l\}_{l=1}^k$ is the standard orthonormal basis of $\mathbb R^k$. Each $\omega_l:\mathbb Z^m\to\mathbb T$ is clearly a character and hence $\omega(p_1+p_2)=\omega(p_1)\omega(p_2)$, $p_1$, $p_2\in \mathbb Z^m$.


\begin{theorem} \label{theorem53}             Let $\Theta\in M_{k+m}$ and $(M_n\otimes A,\mathbb T^{k+m},\alpha\otimes\alpha^A)$ be as above. Then                                 
\[ (M_n \otimes A)^\Theta \simeq M_n \otimes  A^{\Theta_{2,2}}, \]
with the isomorphism given by
\[ \Phi(X \otimes a) 
= \alpha_{\omega(-q)}(\Psi(X))U_{\omega(-2q)} \otimes a, \]
for $X\in M_n$, $a\in A$ homogeneous of order $q\in\mathbb Z^m$ and  $\Psi$  as defined in Lemma \ref{lemma51}
\end{theorem}
\begin{proof}
Let $(\pi,V)$ be a faithful covariant representation of $(A,\mathbb T^m,\alpha^A)$ on $\mathcal H$. Then $(\text{id}\otimes\pi,U\otimes V)$ is a faithful covariant representation of $(M_n\otimes A,\mathbb T^{k+m},\alpha\otimes\alpha^A)$ on $\mathbb C^n\otimes{\mathcal H}$.
Let $(\text{id}\otimes\pi)^\Theta:(M_n\otimes A)^\Theta\to B(\mathbb C^n\otimes\mathcal H)$ and $\pi^{\Theta_{2,2}}:A^{\Theta_{2,2}}\to B(\mathcal H)$ be the $*$-representations defined as in Proposition \ref{theta_rep}.

As $(\text{id}\otimes\pi)^\Theta$ and $\text{id}\otimes \pi^{\Theta_{2,2}}$ are faithful representations of $(M_n\otimes A)^\Theta$ and $M_n\otimes A^{\Theta_{2,2}}$ respectively, to prove the theorem it is enough to show that there exists a unitary operator $W\in B(\mathbb C^n\otimes \mathcal H)$ such that 
\begin{equation}\label{unitary}
W^*(\text{id}\otimes\pi)^{\Theta}(X\otimes a) W=\text{id}\otimes\pi^{\Theta_{2,2}}(\Phi(X\otimes a))
\end{equation}
for all $X\in M_n$, $a\in A_p$ and $p\in\mathbb Z^m$; here $A_p$ is the homogeneous component of order $p$ with respect to $(A,\mathbb T^m,\alpha^A)$.  

Recall the grading on $M_n$ and $A$ which are determined by $(M_n,\mathbb T^k,\alpha)$ and $(A,\mathbb T^m,\alpha^A)$  respectively and the grading on $\mathbb C^n$ and $\mathcal H$ determined by the representations $(U_z)$ and $(V_z)$ of $\mathbb T^k$ and $\mathbb T^m$ respectively.
Let $X\in M_n$ and $a\in A$ be homogeneous of order $p\in\mathbb  Z^k$ and $q\in\mathbb Z^m$ respectively, and let $\xi_1\in \mathbb C^n$ and $\xi_2\in\mathcal H$ be  homogeneous of order $p_1\in \mathbb Z^k$ and $q_1\in\mathbb Z^m$.   
Then $X\otimes a$ is homogeneous of order $(p,q)\in \mathbb Z^k\times\mathbb Z^m$ with respect to $(M_n\otimes A,\mathbb T^{k+m},\alpha\otimes\alpha^A)$ and $\xi_1\otimes\xi_2$ is homogeneous of order $(r,s)\in\mathbb Z^k\times\mathbb Z^m$ with respect to $(U_z\otimes V_z)$. 
Furthermore,  $\alpha_z(X)=\chi_p(z)X$ and $U_z\xi_1=\chi_{p_1}(z)\xi_1$.  Set $b=\left( \begin{array}{c}
    p \\
    q
\end{array} \right)$ and $c= \left( \begin{array}{c}
    r \\
    s
\end{array} \right)$.
Then by Proposition \ref{theta_rep} we have 
\begin{eqnarray*}
&&(\text{id}\otimes\pi)^\Theta(X\otimes a)\xi_1\otimes\xi_2=e^{2\pi i\langle\Theta(b), c\rangle}X\xi_1\otimes \pi(a)\xi_2\\
&& =e^{2\pi i\langle\Theta_{1,1}p,r\rangle} e^{2\pi i\langle\Theta_{2,2}q,s\rangle}
e^{2\pi i\langle\Theta_{2,1}p,s\rangle}
e^{2\pi i\langle\Theta_{1,2}q,r\rangle} X\xi_1\otimes\pi(a)\xi_2\\
&&=e^{2\pi i\langle\Theta_{2,1}p,s\rangle}
e^{-2\pi i\langle\Theta_{2,1}r,q\rangle} \Psi(X)\xi_1\otimes\pi^{\Theta_{2,2}}(a)\xi_2\\
&&= \chi_p(\omega(s))\chi_r(\omega(-q))\Psi(X)\xi_1\otimes\pi^{\Theta_{2,2}}(a)\xi_2\\
&&=\alpha_{\omega(s)}(\Psi(X))U_{\omega(-q)}\xi_1\otimes\pi^{\Theta_{2,2}}(a)\xi_2.
\end{eqnarray*}

Let $W$ be a linear map defined on $\text{span}\{\xi\otimes\eta: \xi\in\mathbb C^n, \eta\in \mathcal H_s, s\in\mathbb Z^m\}\subset \mathbb C^n\otimes\mathcal H$ by letting
$$W(\xi\otimes\eta)=U_{\omega(q)}\xi\otimes\eta, \xi\in\mathbb C^n, \eta\in \mathcal H_q.$$

Any vector $\zeta$ in the span can be written as $\sum_{i=1}^l\xi_i\otimes\eta_i$, where $\xi_i\in\mathbb C^n$ and $\{\eta_i\}_{i=1}^l$ is an orthonormal set in $\mathcal H$ such that $\eta_i\in \mathcal H_{q_i}$, $q_i\in\mathbb Z^m$ ($q_i$ can be equal for different $i$). 
We have  
\begin{eqnarray*}
\langle W\zeta,W\zeta\rangle &=&\left\langle\sum_{i=1}^lU_{\omega(q_i)}\xi_i\otimes\eta_i, \sum_{i=1}^lU_{\omega(q_i)}\xi_i\otimes\eta_i\right\rangle\\&=&\sum_{i=1}^l\langle U_{\omega(q_i)}\xi_i,U_{\omega(q_i)}\xi_i\rangle\langle\eta_i,\eta_i\rangle\\&=&\sum_{i=1}^l\langle\xi_i,\xi_i\rangle\langle\eta_i,\eta_i\rangle=
\langle\zeta,\zeta\rangle,
\end{eqnarray*}
hence $W$ can be extended to an isometry on $\mathbb C^n\otimes \mathcal H$; as the range of $W$ is dense in $\mathbb C^n\otimes \cl H$ it is a unitary operator. 

It is left to see that $W$ satisfies (\ref{unitary}).
For $X$, $a$, $\xi_1$, $\xi_2$ as above, we have
\begin{eqnarray*}
&&W^*(\text{id}\otimes\pi)^{\Theta}(X\otimes a) W\xi_1\otimes\xi_2=W^*(\text{id}\otimes\pi)^{\Theta}(X\otimes a) U_{\omega(s)}\xi_1\otimes\xi_2\\&&=
W^*(\alpha_{\omega(s)}(\Psi(X))U_{\omega(-q)}U_{\omega(s)}\xi_1\otimes\pi^{\Theta_{2,2}}(a)\xi_2\\&&=U_{\omega(s+q)}^*\alpha_{\omega(s)}(\Psi(X))U_{\omega(-q)}U_{\omega(s)}\xi_1\otimes\pi^{\Theta_{2,2}}(a)\xi_2\\&&=
U_{\omega(-q)}\Psi(X)U_{\omega(-q)}\xi_1\otimes\pi^{\Theta_{2,2}}(a)\xi_2\\&&=\alpha_{\omega(-q)}(\Psi(X))U_{\omega(-2q)}\xi_1\otimes\pi^{\Theta_{2,2}}(a)\xi_2\\&&=
\text{id}\otimes\pi^{\Theta_{2,2}}(\Phi(X\otimes a))\xi_1\otimes\xi_2.
\end{eqnarray*}
The result now follows by density arguments.
\end{proof}

We remark that the statements holds true if $M_n$ is replaced by a subalgebra $C$ of $M_n$ such that $\alpha_z(C)=C$ for all $z\in \mathbb T^k$.

\section{$\mathsf{CAR}_{\Theta}$ as Rieffel deformation}
In this section, we will show that our main object  $\mathsf{CAR}_{\Theta}$ can be seen as Rieffel's deformation of a higher dimensional
CAR algebra. 
Recall that the one-dimensional algebra of canonical anti-commutation relations (CAR) is the $*$-algebra  $\mathbb C\langle a,a^*\; |\; a^*a+aa^*=1\rangle$. We will denote its
 universal enveloping $C^*$-algebra by $\mathsf{CAR}_1$; the latter algebra exists
and is isomorphic to a $C^*$-subalgebra of the $C^*$-algebra of all continuous functions on the unit disk $\{z:|z|\leq 1\}$ with values in $M_2$, see e.g. \cite[Theorem 2.2]{CAR1}. Its another realization, which will be convenient for our  purpose, will be described in the next section.  The higher dimensional CAR $C^*$-algebra is given by the tensor product $\mathsf{CAR}_1^{\otimes n}$. Note that  $\mathsf{CAR}_1$ is nuclear, and hence we do not need to specify the $C^*$-tensor product of its copies. 

The $C^*$-algebra $\mathsf{CAR}_1$ has a natural action of $\mathbb T$  given by
\begin{equation}\label{actioncar1}
 \alpha_w(a) = wa, \ w \in \mathbb{T}, 
 \end{equation}
where $a$ is the generator of $\mathsf{CAR}_1$. This $\mathbb T$-action will be always assumed on $\mathsf{CAR}_1$ without mentioning it. It induces the action $\alpha^{\otimes n}$ of $\mathbb T^n$ on the tensor product $\mathsf{CAR}_1^{\otimes n}$ which we also fix through the paper.
For this action each generator $\tilde a_i:=1^{\otimes (i-1)} \otimes a \otimes 1^{\otimes (n - i)}$ is homogeneous of order $\delta_i\in\mathbb Z^n$, where $(\delta_i)_k=\delta_{i,k}$ is the Kronecker delta.

Fix now  a real skew-symmetric matrix $\Theta=(\Theta_{i,j})_{i,j=1}^n$  and recall that $\mathsf{CAR}_{\Theta}$ is the universal  $C^*$-algebra generated by $a_1,\ldots, a_n$ subject to the relations:
\[ a_i^* a_i + a_i a_i^* = 1, \ \]
\[ a_i^* a_j = e^{2 \pi i \Theta_{i,j}} a_j a_i^*, \]
\[ a_i a_j = e^{-2 \pi i \Theta_{i,j}} a_j a_i. \]
We have the following isomorphism between $\mathsf{CAR}_{\Theta}$ and a Rieffel deformation of $\mathsf{CAR}_1^{\otimes n}$: 
\begin{theorem}\label{fiber} Let $\Theta$ be a real skew-symmetric  $n\times n$ matrix. Then 
\[ \mathsf{CAR}_\Theta \simeq (\mathsf{CAR}_1^{\otimes n})^{\frac{\Theta}{2}}. \]
\end{theorem}

\begin{proof}
Consider
\[ \varphi : \mathsf{CAR}_\Theta \rightarrow (\mathsf{CAR}_1^{\otimes n})^{\frac{\Theta}{2}}, \ \varphi(a_i) = 1^{\otimes (i-1)} \otimes a \otimes 1^{\otimes (n - i)}:=\tilde a_i, \ i=1,\ldots, n. \]
We shall see first that $\varphi$ extends to a well-defined $*$-homomorphism. 
As $\tilde a_k$  and $\tilde a_k^*$ are homogeneous of order $\delta_k\in\mathbb Z^n$ and $-\delta_k$ respectively, by Proposition \ref{homogeneous_rieffel} 
\begin{equation*}
    \begin{split}
        \varphi(a_k) \cdot_{\frac{\Theta}{2}} \varphi(a_k)^* & = e^{-\pi i \langle \Theta (\epsilon_k),  \epsilon_k \rangle} \tilde a_k\tilde a_k^*=
        \\ & = 1^{\otimes (k-1)} \otimes aa^* \otimes 1^{\otimes (n - k)} = \varphi(a_ka^*_k). 
    \end{split}
\end{equation*}
Similarly, $\varphi(a_k)^* \cdot_{\frac{\Theta}{2}} \varphi(a_k)=\varphi(a_k^*a_k)$ and hence
$\varphi(a_k)^*\cdot_{\frac{\Theta}{2}}\varphi(a_k)+\varphi(a_k)\cdot_{\frac{\Theta}{2}}\varphi(a_k)^*=1$.

If $k< m$  then 
\begin{equation*}
    \begin{split}
        \varphi(a_k) \cdot_{\frac{\Theta}{2}} \varphi(a_m) & = e^{ \pi i \langle \Theta (\epsilon_k),  \epsilon_m \rangle}\tilde a_k\tilde a_m
        \\ & = e^{- \pi i \Theta_{k,m}} 1^{\otimes (k-1)} \otimes a \otimes 1^{\otimes (m - k - 1)} \otimes a \otimes 1^{\otimes (n - m)}. 
    \end{split}
\end{equation*}
\begin{equation*}
    \begin{split}
        \varphi(a_m) \cdot_{\frac{\Theta}{2}} \varphi(a_k) & = e^{ \pi i \langle \Theta (\epsilon_m), \epsilon_k \rangle}\tilde a_m\tilde a_k
        \\ & = e^{\pi i  \Theta_{k,m}} 1^{\otimes (k-1)} \otimes a \otimes 1^{\otimes {m - k - 1}} \otimes a \otimes 1^{\otimes (n - m)}. 
    \end{split}
\end{equation*}
Thus
\[ \varphi(a_k) \cdot_{\frac{\Theta}{2}} \varphi(a_m) = e^{-2\pi i  \Theta_{k,m}} \varphi(a_m) \cdot_{\frac{\Theta}{2}} \varphi(a_k). \]
Similar calculations give 
\[ \varphi(a_k)^* \cdot_{\frac{\Theta}{2}} \varphi(a_m) = e^{2\pi i  \Theta_{k,m}} \varphi(a_m) \cdot_{\frac{\Theta}{2}} \varphi(a_k)^*, \]
so $\varphi$ extends to a $*$-homomorphism.

 $\varphi$ is surjective as the $*$-algebra generated by $\tilde a_i$, $i=1,\ldots, n$, is dense in $(\mathsf{CAR}_1^{\otimes n})^{\frac{\Theta}{2}}$.

$\mathsf{CAR}_\Theta$ has also a natural $\mathbb T^n$-action determined by $\alpha_w(a_i)=w_ia_i$, $w=(w_1,\ldots,w_n)\in\mathbb T^n$ and hence we can talk about its Rieffel deformation $(\mathsf{CAR}_\Theta)^{-\frac{\Theta}{2}}$.  For $\tilde a_k\in \mathsf{CAR}_1^{\otimes n}$, $k=1,\ldots,n$, as above consider the map
\[ \Psi : \mathsf{CAR}_1^{\otimes n}\rightarrow (\mathsf{CAR}_\Theta)^{-\frac{\Theta}{2}}, \ \Psi(\tilde a_k) = a_k, k=1,\ldots, n. \]
As $a_k$ and $a_k^*\in \mathsf{CAR}_\Theta$ are homogeneous of order $\delta_k$ and $-\delta_k\in\mathbb Z^n$ respectively, as above we obtain that 
\begin{equation*}
    \begin{split}
        & \Psi(\tilde a_k) \cdot_{-\frac{\Theta}{2}} \Psi(\tilde a_m) = {e^{\pi i \Theta_{k,m} }} a_ka_m  = \\ & = { e^{- \pi i \Theta_{k,m} }} a_ma_k = \Psi(\tilde a_m) \cdot_{-\frac{\Theta}{2}}  \Psi(\tilde a_k) .
    \end{split}
\end{equation*}
In a similar way we get $\Psi(\tilde a_k)^*\cdot_{-\frac{\Theta}{2}}\Psi(\tilde a_m)=\Psi(\tilde a_m)\cdot_{-\frac{\Theta}{2}}\Psi(\tilde a_k)^*$, $m\ne k$ and $\Psi(\tilde a_k)^*\cdot_{-\frac{\Theta}{2}}\Psi(\tilde a_k)+\Psi(\tilde a_k)\cdot_{-\frac{\Theta}{2}}\Psi(\tilde a_k)^*=1$. 
Thus $\Psi$ extends to a  $*$-homomorphism. It is clearly surjective.
Moreover, one has the following commutative diagram
\[
\begin{tikzcd}
(\mathsf{CAR}_\Theta)^{-\frac{\Theta}{2}} \arrow[r, "\varphi^{-\frac{\Theta}{2}}"] & ((\mathsf{CAR}_1^{\otimes n})^{\frac{\Theta}{2}})^{-\frac{\Theta}{2}} \\
\mathsf{CAR}_1^{\otimes n} \arrow[u, "\Psi"] \arrow[ur, equal, "\id"]
\end{tikzcd}
\]
Since $\Psi$ is surjective, $\varphi^{-\frac{\Theta}{2}}$ is injective. Therefore, by Proposition \ref{theta_rep} $\varphi$ is injective  too. 

\end{proof}

\section{The $C^*$-algebra $\mathsf{CAR}_1$}
In this section we recall the representation theory of the one-dimensional CAR $*$-algebra,  and describe its
 universal enveloping $C^*$-algebra as a subalgebra of $C(\left[0,\frac{1}{2}\right], M_2(C(\mathbb T))$, showing that it has a $C(\left[0,\frac{1}{2}\right])$-structure and that the action  $\alpha$ of $\mathbb T$ on $\mathsf{CAR}_1$ defined by (\ref{actioncar1}) is fiberwise. 
 
\subsection{Representation theory of $\mathsf{CAR}_1$}
We will use the following  classification of irreducible representations of CAR up to unitary equivalence:

\begin{itemize}
\item 2-dimensional:
$$\pi_{x,\varphi}(a)=e^{i\varphi}\left(\begin{array}{cc} 0&\sqrt{x}\\\sqrt{1-x}&0\end{array}\right), x\in \left[0,\frac{1}{2}\right), \varphi\in [0,\pi),$$
\item 1-dimensional:
$$\rho_\varphi(a)=\frac{e^{i\varphi}}{\sqrt{2}}, \varphi\in[0,2\pi).$$
\end{itemize}
\begin{remark}\rm 
These representations are unitary equivalent to the representations given in \cite{CAR1}.

We note also that
$$\pi_{x,\varphi}(a)= W^*\pi_{x,\varphi-\pi}(a) W, x\in \left[0,\frac{1}{2}\right), \varphi\in [\pi,2\pi),$$
where
$W= \left(\begin{array}{cc} 1&0\\0&-1\end{array}\right)$,
$$\pi_{0,1}(a)= W(\varphi)^*\pi_{0,\varphi}(a) W(\varphi),  \varphi\in [\pi,2\pi),$$
where $W(\varphi)=\left(\begin{array}{cc} 1&0\\0&e^{i\varphi}\end{array}\right)$
and
 $$\pi_{\frac{1}{2},\varphi}(a)=\frac{e^{i\varphi}}{\sqrt{2}}\left(\begin{array} {cc} 0&1\\1&0\end{array}\right)=V\frac{e^{i\varphi}}{\sqrt{2}}\left(\begin{array}{cc} 1&0\\0&-1\end{array}\right)V^*,$$
where $V=\frac{1}{\sqrt{2}}\left(\begin{array}{cc} 1&1\\1&-1\end{array}\right)$.
Hence any one-dimensional irreducible representation can be obtained by decomposing $\pi_{\frac{1}{2},\varphi}$, $\varphi\in[0, \pi)$,  into irreducible ones. Also one has
\[ \pi_{\frac{1}{2}, \varphi} = W \pi_{\frac{1}{2}, \varphi+\pi} W^*, \varphi\in [0,\pi). \]
\end{remark}

\subsection{Spatial picture of $\mathsf{CAR}_{1}$}\label{subsection6.2}

In order to describe the universal enveloping  $C^*$-algebra $\mathsf{CAR}_{1}$,
we recall the following version of the  Stone-Weierstrass-Glimm Theorem, see e.g. \cite[Theorem 1.4]{F1}, \cite[Section 3]{V}.

\begin{theorem}\label{gsw}
Let $Y$ be a compact Hausdorff space and let $A\subseteq B$ be
sub\-algeb\-ras of $C(Y,M_n)$. For every pair ($y_1,
y_2$) of points in $Y$  define
$A(y_1,y_2)$  as %
\[
A(y_1,y_2):=\{(f(y_1),f(y_2))\in M_n\times M_n
\ | f\ \in A\},
\]
and similarly $B(y_1,y_2)$. 
Then
\[
A=B\Leftrightarrow A(y_1,y_2)=B(y_1,y_2),\forall y_1,y_2\in Y.
\]
\end{theorem}

For representations $\pi_1$, $\pi_2$ of a $*$-algebra
$\mathcal{A}$ on Hilbert spaces $\mathcal{H}(\pi_1)$ and
$\mathcal{H}(\pi_2)$ respectively, we write $\mathsf{Hom}(\pi_1,\pi_2)$ for the
space of intertwining operators
\[
\mathsf{Hom}(\pi_1, \pi_2)=\{c\in
B(\mathcal{H}(\pi_2),\mathcal{H}(\pi_1)):\pi_1(a) c=c\pi_2(a),\
a\in\mathcal{A}\}.
\]
We remark that $\mathsf{Hom}(\pi_1,\pi_2)=\{0\}$ iff $\pi_1$, $\pi_2$ are
disjoint, i.e. $\pi_1$, $\pi_2$ do not have unitary equivalent
sub-representations.

For a $*$-algebra $\mathcal{A}\subset B(\mathcal{H})$ we denote by
$\mathcal{A}^\prime$ its commutant, i.e.
\[
\mathcal{A}^\prime=\{c\in B(\mathcal{H}):\ ca=ac,\
a\in\mathcal{A}\}.
\]

\medskip

Let also $W(z)=\left(\begin{array}{cc} 1&0\\0&z\end{array}\right)$, $z\in \mathbb T$, and retain the unitaries $W$ and $V$ from the previous subsection; in particular, $W=W(-1)$. Write $D_2\subset M_2$ for the subalgebra of diagonal matrices.  

\smallskip{}

Let $h:\mathsf{CAR}_{1}\to C\left(\left[0, \frac{1}{2} \right], M_2(C(\mathbb T)\right)$ be a $*$-homomorphism given on the generator $a\in \mathsf{CAR}_{1}$ by
\begin{equation}\label{image}
h(a)(x)(z)=z\left(\begin{array}{cc} 0&\sqrt{x}\\\sqrt{1-x}&0\end{array}\right), \ x \in \left[0, \frac{1}{2} \right], \ z \in \mathbb{T}.
\end{equation}
The next proposition gives a desired realization of
$\mathsf{CAR}_{1}$.

\begin{proposition}\label{prop63}
$\mathsf{CAR}_{1}$ is isomorphic to the $C^*$-algebra
\begin{equation}\label{algebrab}
\mathcal B=\left\{f\in C\left(\left[0, \frac{1}{2} \right], M_2(C(\mathbb T))\right): \begin{array}{l} f(x)(z)=Wf(x)(-z)W^*,\\
x\in \left( 0,\frac{1}{2} \right),\\
f(0)(z)=W(z)f(0)(1)W^*(z),  \\
  f(\frac{1}{2})\in V(D_2\otimes C(\mathbb T))V^*, \\ f(\frac{1}{2})(z) = Wf(\frac{1}{2})(-z)W^*, \\  \forall z\in\mathbb T \end{array}  \right\},
  \end{equation}
 with the map $h$ that implements the $*$-isomorphism. 
\end{proposition}

\begin{proof} Set $\mathcal A=\mathsf{CAR}_{1}$.
Observe first that $h(a)(x)(z)=\pi_{x,\varphi}(a)$ for $z=e^{i\varphi}$; hence $h$ extends to a $*$-homomorphism of $\mathsf{CAR}_{1}$. Moreover, it is easy to see that the image  $h(\mathcal A)$ is in $\mathcal B$.  It follows from the classification of irreducible representations of $\mathcal A$ that
$h:\mathcal A\to \mathcal B$ given by  (\ref{image}) is an isometry. Hence it is sufficient to see that
$h$ is surjective.  Considering  $\mathcal B$ as a $C^*$-subalgebra of $C\left(\left[0, \frac{1}{2} \right]\times\mathbb T, M_2\right)$, by Theorem \ref{gsw}, it is enough to show that for pairs $(x_1,x_2)\in \left[0, \frac{1}{2} \right]^2$ and $(z_1,z_2)\in\mathbb T^2$, $h(\mathcal A)((x_1,z_1),(x_2,z_2))=\mathcal B((x_1,z_1), (x_2,z_2))$.
We will follow the same scheme as in \cite[Theorem 2.2]{CAR1} and prove the equality of the commutants
$h(\mathcal A)((x_1,z_1),(x_2,z_2))'=\mathcal B((x_1,z_1),(x_2,z_2))'$.  
For the notation simplicity we will write $\pi_{x,z}$ instead of $\pi_{x,\varphi}$ if $z=e^{i\varphi}$.

Consider the equivalence relation on
$\left[0, \frac{1}{2} \right] \times\mathbb T$ defined as follows
$$(x_1,z_1)\sim (x_2,z_2), \text{ if either } x_1=x_2 \text{ and } z_1=\pm z_2 \text{ or } x_1=x_2=0$$
and note that $\pi_{x_1,z_1}$ and $\pi_{x_2,z_2}$ are disjoint and hence $\mathsf{Hom}(\pi_{x_1,z_1}, \pi_{x_2,z_2})=\{0\}$ when $(x_1,z_1)\not\sim (x_2,z_2)$.
Therefore,  assuming $(x_1,z_1)\not\sim (x_2,z_2)$ we obtain that
$$h(\mathcal A)((x_1,z_1),(x_2,z_2))'=h(\mathcal A)(x_1)(z_1)'\oplus h(\mathcal A)(x_2)(z_2)'.
$$
As $h(\mathcal A)\subset\mathcal B$,
$\mathcal B((x_1,z_1), (x_2,z_2)'=\mathcal B(x_1)(z_1)'\oplus \mathcal B(x_2)(z_2)'$.

If $x_1=x_2=0$, then $\pi_{x_i,z_i}(b)=W(z_i)\pi_{0,1}(b)W(z_i)^*$, $b\in\mathcal A$,  and one easily  gets that
$$h(\mathcal A)((x_1,z_1),(x_2,z_2))'=\{(\Lambda_{ij}): W(z_i)\Lambda_{ij}W(z_j)^*\in \pi_{0,1}(\mathcal A)', i,j=1,2\}$$
and
$$\mathcal B((x_1,z_1), (x_2,z_2))'=\{(\Lambda_{ij}): W(z_i)\Lambda_{ij}W(z_j)^*\in \mathcal B(0,1)', i,j=1,2\}.$$

If $x_1=x_2$ and $z_1=-z_2$, similarly we obtain that
$$h(\mathcal A)((x_1,z_1),(x_2,z_2))'=\{(\Lambda_{ij}): \Lambda_{i1}, W^*\Lambda_{1j}W\in \pi_{x_1,z_1}(\mathcal A)', i,j=1,2\}$$
and
$$\mathcal B((x_1,z_1), (x_2,z_2))'=\{(\Lambda_{ij}): \Lambda_{i1}, W^*\Lambda_{1j}W\in \mathcal B(x_1,z_1)', i,j=1,2\}.$$

If $x_1=x_2$ and $z_1=z_2$, then $$h(\mathcal A)((x_1,z_1),(x_2,z_2))'=(I\otimes h(\mathcal A)(x_1,z_1))'=M_2\otimes  (h(\mathcal A)(x_1,z_1))'$$  and similarly
$\mathcal B((x_1,z_1), (x_2,z_2))'=M_2\otimes  (\mathcal B(x_1,z_1))'$.
Therefore to prove the statement it is enough to see that
\begin{equation}\label{incl}
\pi_{x,z}(\mathcal A)'\subset \mathcal B(x,z)'
\end{equation}
for any $(x,z)\in \left[0, \frac{1}{2} \right] \times\mathbb T$.

We consider two cases: $x\ne \frac{1}{2}$ and $x=\frac{1}{2}$.

\smallskip
Case 1. $x\ne \frac{1}{2}$. In this case  $\pi_{x,z}$ is irreducible and hence $\pi_{x,z}(\mathcal A)'=\mathbb C I_2$, the inclusion (\ref{incl}) holds trivially.

\smallskip

{\it Case 2.} $x=\frac{1}{2}$. In this case we have that $C\in  \pi_{x,z}(\mathcal A)'$ if and only if
$V^*CV\in D$, where $D$ is the subalgebra of the diagonal matrices. it follows from the definition of $\mathcal B$ that any such $C$ is in $\mathcal B\left(\frac{1}{2},z\right)$. This completes the proof.

\end{proof}

\subsection{$\mathsf{CAR}_1$ as a fixed point subalgebra}
It will be useful to see $\mathsf{CAR}_1$ as a fixed point subalgebra of a larger $C^*$-subalgebra of $C\left(\left[0, \frac{1}{2} \right], M_2(C(\mathbb T))\right)$ with a $\mathbb Z_2$-action defined on it. 

Given a $C^*$-algebra $\mathcal A$, the $C^*$-algebra $C(\mathbb T,\mathcal A)$ has a natural $\mathbb T$-action 
which will be always denoted by $\beta$:
$$\beta_w(f)(z)=f(wz), f\in C(\mathbb T,\mathcal A), z\in \mathbb T, w\in\mathbb T.$$\
Considering $\mathbb Z_2$ as the subgroup $\{1,-1\}$ of $\mathbb T$, we shall denote by $\beta$ also the restriction of it to $\mathbb Z_2$.

We keep the notation of subsection \ref{subsection6.2} and consider the $C^*$-algebra
$$\mathcal C=\left\{f\in C\left(\left[ 0, \frac{1}{2} \right], M_2(C(\mathbb T))\right): \begin{array}{l} 
f(0)(z)=W(z)f(0)(1)W^*(z),  \\
  f(\frac{1}{2})\in V(D_2\otimes C(\mathbb T))V^*, \\  \forall z\in\mathbb T\end{array}  \right\}.$$
Let $\sigma$ be the action of $\mathbb{Z}_2$ on $M_2(C(\mathbb{T})) \simeq M_2 \otimes C(\mathbb{T})$ given by
\[\sigma_w = \mathsf{Ad}(W(w)) \otimes \beta_{w}, w\in\mathbb Z_2, \]
where we write $\mathsf{Ad}(v)$ for the inner automorphism $T\mapsto vTv^*$ of  $M_2$. 
Let $\Sigma$ be the action of $\mathbb{Z}_2$ on $\mathcal{C}$ given by
\[ \Sigma_w(f)(x) = \sigma_w(f(x)), \ f \in \mathcal{C}, w\in \mathbb Z_2 \]
\begin{theorem}\label{algebrac} Let $\mathcal B$ be the $C^*$-algebra given by (\ref{algebrab}). Then 
\[ \mathcal{B} \simeq \mathcal{C}^\Sigma. \]
\end{theorem}
\begin{proof}
The only condition to be checked is that the $0$-fiber is stable under $\Sigma$:
\begin{equation*}
    \begin{split}
        \Sigma_{-1}(f)(0)(z) & = W(-1) f(0)(-z) W(-1)^* = \\ & 
                          = W(-1) W(-z) f(0)(1) W(-z)^* W(-1)^* = \\ &
                          = W(z) f(0)(1) W(z)^* = \\ &
                          = f(0)(z).
    \end{split}
\end{equation*}
\end{proof}

\subsection{$C\left(\left[0,\frac{1}{2}\right]\right)$-structure on $\mathsf{CAR}_1$.}\label{taction}
Let $\mathcal B$ be as in Proposition \ref{prop63}. 
As $\mathsf{CAR}_1\simeq \mathcal B$ it has a $C\left(\left[0,\frac{1}{2}\right]\right)$-action induced by the natural  $C\left(\left[0,\frac{1}{2}\right]\right)$-action on $\mathcal B$ given by 
\begin{equation}\label{cstructure}
 \Phi(g)(f)(x) = g(x)f(x), g\in C\left(\left[0,\frac{1}{2}\right]\right),  f\in\mathcal B, 
 \end{equation}
so that $\mathsf{CAR}_1(x)\simeq \cl B(x)$ with the isomorphism defined by $b(x)\mapsto h(b)(x)$, 
$b\in \mathsf{CAR}_1$, $x\in \left[0,\frac{1}{2}\right]$.

\smallskip{}

We next identify the fibers $\mathsf{CAR}_1(x)$, $x\in \left[0,\frac{1}{2}\right]$. 

\smallskip{}

We note first that with $\beta$  as in the previous subsection, we have that  
$w\mapsto\beta(w)(f)(x):=(\beta_w(f(x))$, $f\in C\left(\left[0,\frac{1}{2}\right], M_2(C(\mathbb T)\right)$, $w\in\mathbb T$, is a fiberwise action of $\mathbb T$.   
Moreover, the isomorphism $h$ of Proposition \ref{prop63} is equivariant in the sense that
\[ h(\alpha_w(b))(x) = \beta_w(h(b)(x)), b\in \mathsf{CAR}_1,  w\in\mathbb T, x\in \left[0,\frac{1}{2}\right].\]
In particular it implies that $\alpha$ is fiberwise with the induced action $\alpha^x$ on $\mathsf{CAR}_1(x)$ given by $\alpha_w^x(a(x))=wa(x)$, where $a$ is the generator of $\mathsf{CAR}_1$. 

\medskip


Let $C(\mathbb T)\rtimes_{\beta}{\mathbb Z}_2$ be the crossed product $C^*$-algebra corresponding to the dynamical system $(C(\mathbb T), \beta, \mathbb Z_2)$. It is the universal $C^*$-algebra generated by unitaries $u$ and $v$ satisfying the relations $uv=-vu$, $v^2=1$; the action $\beta$  of $\mathbb T$ 
 on $C(\mathbb T)$, $\beta_w(f)(z)=f(wz)$, $w,z\in \mathbb T$, induces a $\mathbb T$-action $\tau$ on $C(\mathbb T)\rtimes_{\beta}{\mathbb Z}_2$, given by $\tau_w(u)=wu$ and $\tau_w(v)=v$, $w\in\mathbb T$.  
 
 \smallskip{}
 
In what follows it will be convenient to use the generators of Clifford algebra. We recall that  the Clifford $C^*$-algebra $Cl_2$ is
$$Cl_{2}=C^*(e: e^2=0, ee^*+e^*e=1).$$
Clearly, $M_2\simeq Cl_2$, where the isomorphism is given by $\left( \begin{array}{cc}
            0 & 0  \\
            1 & 0
        \end{array} \right) \mapsto e$. 

\begin{theorem}\label{theorem65}
One has the following isomorphisms of fibers of $\mathsf{CAR}_1$:
\begin{equation*}
    \begin{split}
        & \mathsf{CAR}_1(0) \simeq M_2 \simeq Cl_2, \ \psi_0:h(a)(0) \mapsto\left( \begin{array}{cc}
            0 & 0  \\
            1 & 0
        \end{array} \right) \mapsto e,  \\
        & \mathsf{CAR}_1(x) \simeq C(\mathbb{T}) \rtimes_\beta \mathbb{Z}_2, \\
        &\ \psi_x:h(a)(x)\mapsto \frac{u}{2}\left((\sqrt{1-x} + \sqrt{x})1 + (\sqrt{1-x}-\sqrt{x})v\right), \ 0 < x < \frac{1}{2}, \\
        & \mathsf{CAR}_1\left(\frac{1}{2}\right) \simeq C(\mathbb{T}), \ \psi_{\frac{1}{2}}:h(a)\left(\frac{1}{2}\right)(z) \mapsto \frac{1}{\sqrt{2}}z.
    \end{split}
\end{equation*}
Moreover, the isomorphisms are $\mathbb{T}$-equivariant when $M_2$, $C(\mathbb{T}) \rtimes_\beta \mathbb{Z}_2$, $C(\mathbb{T})$ are equipped with the $\mathbb T$-action  given by 
\[ w \curvearrowright T = W(w)T W(w)^*, T\in M_2,\]
\[ w \curvearrowright T =  \tau_w(T), T\in C(\mathbb{T}) \rtimes_\beta \mathbb{Z}_2 ,\]
\[ w \curvearrowright f = \beta_w(f), f\in C(\mathbb T). \]
\end{theorem}
\begin{proof}
For $x \in \left(0,\frac{1}{2}\right)$, by Theorem \ref{algebrac}, we have
\[ \mathsf{CAR}_1(x) \simeq (\mathcal{C}^\Sigma)(x) \simeq (\mathcal{C}(x))^\sigma = M_2(C(\mathbb{T}))^\sigma, \]
with the natural $C(\left[0,\frac{1}{2}\right])$-structure on $\mathcal C$ given as in (\ref{cstructure}). By the duality Theorem (see e.g. \cite[Section 7.1]{williams}), $M_2\otimes C(\mathbb T)\simeq (C(\mathbb{T}) \rtimes_\beta \mathbb{Z}_2 )\rtimes_{\widehat{\beta}}\mathbb Z_2$,
where $\widehat{\beta}$ is the dual action of $\widehat{\mathbb Z_2}\simeq\mathbb Z_2$ on $C(\mathbb{T}) \rtimes_\beta \mathbb{Z}_2$; the double dual  action $\widehat{\widehat{\beta}}$ on $(C(\mathbb{T}) \rtimes_\beta \mathbb{Z}_2 )\rtimes_{\widehat{\beta}} \mathbb{Z}_2)$ is carried by the isomorphism into $\tilde\beta$ on $M_2\otimes C(\mathbb T)$ given by $\tilde\beta_w=\text{Ad}(W(w))\otimes\beta_w=\sigma_w$, $w\in\mathbb Z_2$ ($w\mapsto W(w)$ is unitary equivalent to the left regular representation of $\mathbb Z_2$); hence, using the fixed point theorem, we obtain 
\[ M_2(C(\mathbb{T}))^\sigma\simeq ((C(\mathbb{T}) \rtimes_\beta \mathbb{Z}_2 )\rtimes_{\widehat{\beta}} \mathbb{Z}_2)^{\widehat{\widehat{\beta}}} \simeq C(\mathbb{T}) \rtimes_\beta \mathbb{Z}_2 \] 
 In particular, the isomorphism maps the generators $u$ and $v$ of  $C(\mathbb T)\rtimes_\beta\mathbb Z_2$ to  $z\left( \begin{array}{cc}
    0 & 1 \\
    1 & 0
\end{array} \right)$ and $\left( \begin{array}{cc}
    1& 0 \\
    0 & -1
\end{array} \right)$  in $M_2(C(\mathbb T))^\sigma$ respectively, and from which 
one can easily see that the element $\displaystyle\frac{u}{2}((\sqrt{1-x} + \sqrt{x})1 + (\sqrt{1-x}-\sqrt{x})v)$ maps to $z\left( \begin{array}{cc}
    0 & \sqrt{x} \\
    \sqrt{1-x} & 0
\end{array} \right)$.

 If $x = 0$, then
\[ \mathsf{CAR}_1(0)\simeq \{ f \in M_2(C(\mathbb{T})) : f(z) = W(z) f(1) W(z)^*, z\in\mathbb T\} \simeq M_2(\mathbb{C}), \]
with the isomorphism given by $f\mapsto f(1)$.

For $x = \frac{1}{2}$ one has the isomorphism $\mathcal{C}(\frac{1}{2}) \simeq C(\mathbb{T}) \oplus C(\mathbb{T})$ given by $\phi: f \mapsto V^* f V$, where $V=\frac{1}{\sqrt{2}}\left( \begin{array}{cc}
    1 & 1 \\
    1 & -1
\end{array} \right)$. Let $X = \left( \begin{array}{cc}
    0 & 1 \\
    1 & 0
\end{array} \right)$ and $\widetilde{\sigma}$ be the action of $\mathbb{Z}_2$ on $M_2(C(\mathbb{T}))$ given by
\[ \widetilde{\sigma}(f)(z) = X f(-z) X^*, z\in\mathbb T. \]
Then
\[ \phi(\sigma(f)) = \widetilde{\sigma}(\phi(f)), \ f \in M_2(C(\mathbb{T})). \]
Notice that $\widetilde{\sigma}$ acts on $C(\mathbb{T}) \oplus C(\mathbb{T})$ as
\[ \widetilde{\sigma}(f,g)(z) = (g,f)(-z), z\in\mathbb T. \]
Hence
\[ \mathsf{CAR}_1\left(\frac{1}{2}\right) \simeq (C(\mathbb{T}) \oplus C(\mathbb{T}))^{\widetilde{\sigma}} \simeq C(\mathbb{T}). \]
The formula for $\psi_{\frac{1}{2}}$ can be easily derived.  That the isomorphisms are $\mathbb T$-equivariant is straightforward. 
\end{proof}

\begin{remark}\label{anotheriso}\rm
In what follows we shall also use the isomorphism $\mathsf{CAR}_1(0) \simeq M_2$ given by $h(a)(0)\mapsto\left(\begin{array}{cc}0&1\\0&0
\end{array}\right)$. The isomorphism is $\mathbb T$-equivariant when $M_2$ is given the $\mathbb T$-action $ w \curvearrowright T = W(w)^*T W(w), T\in M_2, w\in\mathbb T$.
\end{remark}

\section{$\mathsf{CAR}_{\Theta}$ as  $C(K_n)$-algebra and its fibers}
Let $K_n = \left[0, \frac{1}{2} \right]^n$. We shall now use our knowledge about $\mathsf{CAR}_1$ to describe a $C(K_n)$-structure on $\mathsf{CAR}_{\Theta}$  and the corresponding fibers.  
For this we will use the fact that $\mathsf{CAR}_{\Theta}$ is Rieffel's deformation of the tensor product of $n$ copies of  $\mathsf{CAR}_1$ (Theorem \ref{fiber}). The result will allow us in particular to obtain a classification of all irreducible representations of $\mathsf{CAR}_{\Theta}$ providing an alternative proof of \cite[Theorem 3]{proskurin-sukretnyi}. 

\smallskip

Let $\Theta$ be a real skew-symmetric $n\times n$ matrix and let $\alpha$ be  
 the action of $\mathbb{T}$ on $\mathsf{CAR}_1$ given by (\ref{actioncar1}). 
Since $\alpha$ is fiberwise with respect to the $C(\left[0,\frac{1}{2}\right])$-structure on $\mathsf{CAR}_1$,  we get an action on $\mathsf{CAR}_1(x)$. Similarly, $\alpha^{\otimes n}$ is fiberwise with respect to the natural $C(K_n)$-structure on $\mathsf{CAR}_1^{\otimes n}$. 
Thus by Theorem \ref{fiber} and Proposition \ref{cx_deformation} one has the following statement:
\begin{proposition}\label{fiber_struct_CAR}
There exists  a $C(K_n)$-structure on $\mathsf{CAR}_\Theta$ such that
\[ \mathsf{CAR}_\Theta(x) \simeq (\mathsf{CAR}_1^{\otimes n}(x))^{\frac{\Theta}{2}}. \]
\end{proposition}

Next we shall give a more explicit description of the fibers. 
\medskip{}

Given $x=(x_1, \ldots, x_n) \in K_n$, let
\begin{equation*}
    \begin{split}
        & L_x = \{i \in \mathbb N_n : x_i = 0 \},  \\ 
        & M_x = \left\{ i \in \mathbb N_n: 0 < x_i < \frac{1}{2} \right\}, \\
        & R_x = \left\{ i \in \mathbb N_n : x_i = \frac{1}{2} \right\}.
    \end{split}
\end{equation*}

For $S=\{S(1),\ldots, S(m)\} \subset \{1, \ldots, n \}$, we let $\Theta_S$ to be the  $m\times m$ matrix such that $(\Theta_S)_{i,j} = \Theta_{S(i), S(j)}$.
For a set $Y$, we write $Y^S=\{(a_i)_{i\in S}: a_i\in Y \}$. If $Y$ is a group, then so is  $Y^S$ with respect to coordinate-wise multiplication; similarly, $Y^S$ is a Hilbert space if so is $Y$ with natural linear operations and scalar product on it.   
Set   $$Cl_S=C^*(e_i, i\in S :  e_i^2 = 0, \ e_i e_i^* + e_i^* e_i = 1, \ e_i e_j = e_j e_i) \simeq Cl_2^{\otimes |S|}$$
and  write $C(\mathbb T^S_{\Theta_S})$ for the non-commutative torus: $$C^*(u_k, k\in S: u_ku_l=e^{-2\pi i\Theta_{k,l}}u_lu_k , u_ku_k^*=u_k^*u_k=1).$$
If $S=\{1,\ldots,n\}$ we write simply $Cl_{2n}$ and $C(\mathbb T^n_\Theta)$; if $n=2$, identifying $\Theta$ with $\theta:=\Theta_{1,2}\in\mathbb R$, we denote the non-commutative torus by  $C(\mathbb T^2_{\theta})$.

\medskip

For $x \in K_n$, let $\beta_\Theta^x$ be the action of $\mathbb Z_2^{M_x}$ on $C(\mathbb{T}^{M_x\sqcup R_x}_{\Theta_{M_x \sqcup R_x}})$ given by 
 $$\beta_\Theta^x(\omega)(u_l)=\omega_lu_l, l\in M_x \text{ and }\beta_\Theta^x(\omega)(u_l)=u_l, l\in R_x,$$ for $\omega=(\omega_k)_{k\in M_x}\in \mathbb Z_2^{M_x}$.
 Then $C(\mathbb{T}^{M_x\sqcup R_x}_{\Theta_{M_x \sqcup R_x}}) \rtimes_{\beta_\Theta^x}\mathbb{Z}_2^{M_x}$ is generated by $u_i$, $i\in M_x\sqcup R_x$, which satisfy the relations in  $C(\mathbb{T}^{M_x\sqcup R_x}_{\Theta_{M_x \sqcup R_x}})$, and self-adjoint unitaries $v_i$, $i\in M_x$,  such that $$v_iv_j=v_jv_i \text{ and } v_i^*u_jv_i=\beta_\Theta^x(\omega^i)(u_j),$$ $i\in M_x$, $j\in M_x\sqcup R_x$, where $\omega^i_i=-1$ and $\omega^i_k=1$ otherwise.

\begin{proposition}\label{isomorphism}
Let $x=(x_1, \ldots, x_n) \in K_n$.  Then
\[ \mathsf{CAR}_\Theta(x) \simeq Cl_{L_x} \otimes C\left(\mathbb{T}^{M_x\sqcup R_x}_{\Theta_{M_x \sqcup R_x}}\right) \rtimes_{\beta_\Theta^x}\mathbb{Z}_2^{M_x}. \]
The isomorphism is given by 
\begin{equation*}
    \begin{split}
       & h^x_\Theta(a_i(x)) = \prod_{k\in L_x} (e_k e_k^* + e^{\pi i \Theta_{i,k}} e_k^* e_k) e_i \otimes 1, i \in L_x, \\
       & h^x_\Theta(a_i(x)) = \prod_{k\in L_x} (e_k e_k^* + e^{2\pi i \Theta_{i,k}} e_k^* e_k) \otimes  \frac{u_i}{2}\left((\sqrt{1-x_i}+\sqrt{x_i})\right. \\&\qquad\qquad\qquad+ \left.(\sqrt{1 - x_i}-\sqrt{x_i}) v_i\right) , i \in M_x, \\
& h^x_\Theta(a_i(x)) = \prod_{k \in L_x} (e_ke_k^* + e^{2\pi i \Theta_{i,k}} e_k^* e_k) \otimes \frac{u_i}{\sqrt{2}}  \ ,  i \in R_x.
    \end{split}
\end{equation*}
\end{proposition}
\begin{proof}
By Proposition \ref{fiber_struct_CAR} we have 
\begin{equation}\label{iso_1}
        \mathsf{CAR}_\Theta(x) \simeq (\mathsf{CAR}_1^{\otimes n}(x))^{\frac{\Theta}{2}} \simeq (\bigotimes_{i = 1}^n \mathsf{CAR}_1(x_i))^{\frac{\Theta}{2}}. 
\end{equation} 
By Theorem \ref{theorem65}, 
\begin{equation}\label{iso_2}
       \mathsf{CAR}_\Theta(x) \simeq \left(Cl_{L_x} \otimes (C(\mathbb{T}^{M_x}) \rtimes_\beta \mathbb{Z}_2^{M_x}) \otimes C(\mathbb{T}^{R_x})\right)^{\frac{\Theta}{2}},
\end{equation}
where the action of $\mathbb T^{L_x\sqcup M_x\sqcup R_x}$, which determines the latter Riffel deformation, is given by the corresponding product of $\mathbb T$-actions on $Cl_2\simeq M_2$, $C(\mathbb T)\rtimes_\beta\mathbb Z^2$ and $C(\mathbb T)$ given in Remark \ref{anotheriso} (for the action on $M_2$) and Theorem \ref{theorem65}.  Identify $Cl_{L_x}$ with $\otimes_{k\in L_x}M_2$ through $e_k\mapsto\otimes_{l\in L_x} a_l$, where $a_k=\left(\begin{array}{cc}0&1\\0&0\end{array}\right)$ and $a_l=I_2$ otherwise.  The action $\mathbb T^{L_x}$ on $Cl_{L_x}$, given by  \[ \alpha_{L_x}: (z_i)_{i\in L_x} \curvearrowright e_i = z_i e_i, \]
is implemented by the unitary representation $(U_z)$ of $\mathbb T^{L_x}$ on $(\mathbb{C}^2)^{L_x}$ given by 
\begin{equation*}
    \begin{split}
       &   U_z\left(\bigotimes_{i\in L_x}\left( \begin{array}{c}
     \xi_i^1  \\
     \xi_i^2 
\end{array} \right)\right) = \bigotimes_{x\in L_x}\left( \begin{array}{c}
      \xi_i^1  \\
     \overline{z_i} \xi_i^2 
\end{array} \right), \quad z=(z_i)_{i\in L_x}\in\mathbb T^{L_x} .
    \end{split}
\end{equation*}
By Theorem \ref{theorem53} and the remark after it, we have 
\begin{equation}\label{iso_3}
\mathsf{CAR}_\Theta(x) \simeq Cl_{L_x} \otimes \left((C(\mathbb{T}^{M_x}) \rtimes_\beta \mathbb{Z}_2^{M_x})\otimes C(\mathbb{T}^{R_x})\right)^{\frac{\Theta_{M_x \sqcup R_x}}{2}}.
\end{equation}
Since $\tilde\beta:=\beta\otimes{\rm id}:{\mathbb Z_2}^{M_x\sqcup R_x}\to {\rm Aut}(C(\mathbb{T}^{M_x})\otimes C(\mathbb{T}^{R_x}))$  commutes with the action that defines the Rieffel deformation, by Remark 3.6 and Proposition 3.3 we have 
\begin{equation*}
    \begin{split}
        \mathsf{CAR}_\Theta(x) & \simeq Cl_{L_x} \otimes \left(C(\mathbb{T}^{M_x 
        \sqcup R_x}) \rtimes_{\tilde \beta} \mathbb{Z}_2^{M_x}\right)^{\frac{\Theta_{M_x \sqcup R_x}}{2}}  \simeq \\ & \simeq Cl_{L_x} \otimes \left(C(\mathbb{T}^{M_x 
        \sqcup R_x})\right)^{\frac{\Theta_{M_x \sqcup R_x}}{2}} \rtimes_{\beta_\Theta^x} \mathbb{Z}_2^{M_x}. \\ & \simeq Cl_{L_x} \otimes \left(C(\mathbb{T}^{M_x 
        \sqcup R_x}_{\Theta_{M_x \sqcup R_x}}) \rtimes_{\beta_\Theta^x} \mathbb{Z}_2^{M_x}\right).
    \end{split}
\end{equation*}

To see the formulas recall the maps $\Phi$ and $\Psi$ from Section \ref{riefell_mna}.  
If $i \in L_x$  then
\[ h_\Theta^x(a_i(x)) = \Phi( e_i \otimes 1) = \Psi(e_i) \otimes 1. \]
Let $\xi \in (\mathbb{C}^2)^{L_x}$ be homogeneous of order $q = (b_i)_{i\in L_x}$ with $b_i \in \{0, -1\}$,  i.e.  $\xi=\otimes_{i\in L_x}f_{b_i}$, where $\left\{f_0=\left(\begin{array}{c} 1\\0\end{array}\right), f_{-1}=\left(\begin{array}{c} 0\\1\end{array}\right)\right\}$ is the standard basis in $\mathbb C^2$.   If $b_k =-1$ then $e_k^* e_k (\xi) = \xi, \ e_k e_k^*(\xi) = 0$. Since $e_j$ is homogeneous of order $p = \delta_j\in \mathbb Z^{L_x}$,  one has
\begin{equation*}
    \begin{split}
        \Psi(e_i)\xi & = e^{2 \pi i \langle \Theta_{L_x}(\delta_i)/2, q \rangle} e_i \xi =  e^{2 \pi i \sum_{k \in L_x} \frac{1}{2}\Theta_{k,i} b_k} e_i \xi \\ &= \prod_{k:b_k = -1} e^{- \pi i \Theta_{k,i}} e_i \xi = \\ & = \prod_{k \in L_x}(e_k e_k^* + e^{-\pi i \Theta_{k,i}} e_k^* e_k) e_i \xi.
    \end{split}
\end{equation*}

Let $i \in R_x$.   Since $u_i$ is homogeneous of order $\delta_i\in \mathbb Z^{M_x\sqcup R_x}$, by Theorem \ref{theorem53}
\[ h_\Theta^x(a_i(x)) = (\alpha_{L_x})_{\omega(-\delta_i)}(\Psi(1))U_{\omega(-2\delta_i)} \otimes \frac{1}{\sqrt{2}} u_i =  U_{\omega(-2\delta_i)} \otimes \frac{1}{\sqrt{2}} u_{i}. \]
 Notice that
\begin{equation*}
U_z = \prod_{k \in L_x} (e_k e_k^* + \overline{z_k} e_k^* e_k), \ z= (z_i)_{i\in L_x},
\end{equation*}
and $\omega_{j}(-\delta_i) = e^{-2 \pi i \Theta_{i, j}}$. Thus 
\[  h_\Theta^x(a_i(x)) =\prod_{k \in L_x} (e_k e_k^* + e^{2\pi i \Theta_{i,k}} e_k^* e_k)  \otimes \frac{1}{\sqrt{2}} u_{i}. \]
If $i \in M_x$ then similar calculations give
\begin{eqnarray*}
 h_\Theta^x(a_i(x)) =&& \prod_{k \in L_x} (e_k e_k^* + e^{2\pi i \Theta_{i,k}} e_k^* e_k)  \otimes \frac{u_i}{2}((\sqrt{1-x_i}+\sqrt{x_i})\\&& + (\sqrt{1 - x_i}-\sqrt{x_i})v_i).  
\end{eqnarray*}
\end{proof}

Now, having a description of fibers of $\mathsf{CAR}_{\Theta}$, we can classify all irreducible representations of $\mathsf{CAR}_{\Theta}$. The following lemma reduces the procedure to the classification of irreducible representations of the fibers.
\begin{lemma}\cite[Proposition C.5]{williams}\label{rep}
Suppose a $C^*$-algebra $\mathcal A$ is equipped with a $C_0(X)$-structure. Then any irreducible representation of $\mathcal A$ factors through an irreducible representation of a fiber $\mathcal A(x)$ for some $x\in X$. 
\end{lemma}

 Lemma \ref{rep} and Proposition \ref{isomorphism} reduce  the classification of all irreducible representations of $\mathsf{CAR}_{\Theta}$ to that of the $C^*$-algebra $Cl_{2k} \otimes C(\mathbb{T}^{n+m}_\Theta) \rtimes_{\beta_\Theta} \mathbb{Z}_2^n$. We will next derive explicit formulas reducing further the classification to the classification of irreducible representations of a non-commutative torus. 

As in the proof of Proposition \ref{isomorphism} we write $e_i$ for the image of $e_i\in Cl_S$ in $\otimes_{k\in S}M_2$ for $S=L_x$ and $S=M_x$, $x\in \left[0,\frac{1}{2}\right]^{M_x}$, i.e. $e_i=\otimes_{k\in S} a_k$ with $a_i=\left(\begin{array}{cc}0&1\\0&0\end{array}\right)$  and $a_k=I_2$ otherwise.

\begin{lemma}\label{reprholemma}
Any irreducible representation of $\mathsf{CAR}_{\Theta}$ is unitary equivalent to a subrepresentation of the representation $\rho_x\circ{\rm ev}_x$, $x=(x_i)_{i\in M_x}\in \left[0,\frac{1}{2}\right]^{M_x}$, where
${\rm ev}_x: \mathsf{CAR}_{\Theta}\to \mathsf{CAR}_{\Theta}(x)$, $a\mapsto a(x)$ and $\rho_x$ is the representation of $\mathsf{CAR}_{\Theta}(x)$ on $(\otimes_{k\in L_x}\mathbb C^2)\bigotimes (\otimes_{k\in M_x}\mathbb C^2)\bigotimes H$ given by 
\begin{equation}\label{reprho}
    \begin{split}
       & \rho_x(a_i(x)) = \prod_{k\in L_x} (e_k e_k^* + e^{\pi i \Theta_{i,k}} e_k^* e_k) e_i \otimes 1, i \in L_x, \\
       & \rho_x(a_i(x)) = \prod_{k\in L_x} (e_k e_k^* + e^{2\pi i \Theta_{i,k}} e_k^* e_k) \otimes  (\sqrt{x_i}e_i + \sqrt{1 - x_i} e_i^*)\otimes u_i ,  i \in M_x, \\
& \rho_x(a_i(x)) = \prod_{k \in L_x} (e_k e_k^* + e^{2\pi i \Theta_{i,k}} e_k^* e_k) \otimes 1\otimes\frac{1}{\sqrt{2}} u_i \ ,  i \in R_x.
    \end{split}
\end{equation}
where $\{u_i: i\in M_x\sqcup R_x\}$ is an irreducible representation of $C(\mathbb T^{M_x\sqcup R_x}_{\Theta_{M_x\sqcup R_x}})$ on the Hilbert space $H$. 
\end{lemma}
\begin{proof}
Proposition \ref{isomorphism} and the duality arguments for crossed products as in the proof of Theorem \ref{theorem65} gives
\begin{eqnarray*}
\mathsf{CAR}_{\Theta}(x)&\simeq& Cl_{L_x}\otimes \left(M_{2|M_x|}\otimes C(\mathbb T_{\Theta_{M_x\sqcup R_x}}^{M_x\sqcup R_x})\right)^{\tilde\beta_{\Theta_{M_x\sqcup R_x}}}\\&\subset& \left(\otimes_{k\in L_x} M_2\right)\bigotimes\left(\otimes_{k\in M_x} M_2\right)\bigotimes C(\mathbb T_{\Theta_{M_x\sqcup R_x}}^{M_x\sqcup R_x}),
\end{eqnarray*}
where $\tilde\beta_{\Theta_{M_x\sqcup R_x}}$ is defined by $$\tilde\beta_{\Theta_{M_x\sqcup R_x}}(w)=\text{Ad}(W(w_i))^{\otimes |M_x|}\otimes\beta_\Theta^x(w), w=(w_i)_{i\in M_x})\in \mathbb Z_2^{M_x}.$$
The imbedding is given by (\ref{reprho}) with $(u_i)_{i\in M_x\sqcup R_x}$ being the generators of $C(\mathbb T_{\Theta_{M_x\sqcup R_x}}^{M_x\sqcup R_x})$.
As any irreducible representation of  $$\left(\otimes_{k\in L_x} M_2\right)\bigotimes\left(\otimes_{k\in M_x} M_2\right)\bigotimes C(\mathbb T_{\Theta_{M_x\sqcup R_x}}^{M_x\sqcup R_x})$$ is unitary equivalent to $\text{id}\otimes\text{id}\otimes\pi$, where $\pi$ is a representation of $C(\mathbb T_{\Theta_{M_x\sqcup R_x}}^{M_x\sqcup R_x})$, the statement now follows from \cite[Proposition 2.10.2]{dixmier}. 
 \end{proof}

The next result was proved in \cite{proskurin-sukretnyi}, but here  we present its alternative proof that uses essentially new approach employing $C(K_n)$-structure of  $\mathsf{CAR}_{\Theta}$.  The representations in the list below are unitary equivalent to those given in \cite[Theorem 3]{proskurin-sukretnyi}. 
\begin{theorem}\label{clas_irr}
Any irreducible representation of  $\mathsf{CAR}_{\Theta}$ is unitary equivalent to a representation $\tau_x$, $x\in [0,\frac{1}{2}]^{M_x}$ given  on $(\otimes_{k\in L_x}\mathbb C^2)\bigotimes (\otimes_{k\in M_x}\mathbb C^2)\bigotimes H$ by
\begin{equation}\nonumber
    \begin{split}
       & \tau_x(a_i) = \prod_{k\in L_x} (e_k e_k^* + e^{\pi i \Theta_{i,k}} e_k^* e_k) e_i \otimes 1, i \in L_x, \\
       & \tau_x(a_i) = \prod_{k\in L_x} (e_k e_k^* + e^{\pi i \Theta_{i,k}} e_k^* e_k) \otimes \left(\left(\prod_{k\in M_x, k<i} ( e_k^* e_k + e^{2\pi i \Theta_{i,k}} e_k e_k^*)\otimes 1_H\right)\right. \\&\times\left.\left(\sqrt{x_i}\prod_{k\in M_x,k\geq i}(e_k^* e_k + e^{4\pi i \Theta_{i,k}} e_k e_k^*)e_i\otimes v_i + \sqrt{1 - x_i} e_i^*\otimes 1_H\right)\right) ,  i \in M_x, \\
& \tau_x(a_i) = \prod_{k \in L_x} (e_k e_k^* + e^{\pi i \Theta_{i,k}} e_k^* e_k) \otimes \prod_{k\in M_x}(e_k^* e_k + e^{2\pi i \Theta_{i,k}} e_k e_k^*)\otimes\frac{v_i}{\sqrt{2}}\ ,  i \in R_x.
    \end{split}
\end{equation}
where $(v_i)_{i\in M_x\sqcup R_x}$ defines an irreducible representation of $C(\mathbb T^{M_x\sqcup R_x}_{\Sigma})$ on $H$,  where $$\Sigma_{i,j}=\left\{\begin{array}{cl}4\Theta_{i,j}& i,j\in M_x,\\2\Theta_{i,j}& (i,j)\text{ or }(j,i)\in M_x\times R_x,\\
\Theta_{i,j}& i,j\in R_x.\end{array}\right.$$ 
Moreover, two such irreducible representations $\tau_x$ and $\tau_y$ are unitary equivalent if and only if $x=y$ and the corresponding representations of  $C(\mathbb T^{M_x\sqcup R_x}_{\Sigma})$ are unitary equivalent. 
\end{theorem}

\begin{proof}
Recall the representation $\rho_x$ from Lemma \ref{reprholemma} and consider  the unitary operators on $\left(\otimes_{k\in L_x}\mathbb C^2\right)\bigotimes \left(\otimes_{k\in M_x}\mathbb C^2\right)\bigotimes H$ given by
$$V_k= 1\otimes(e_ke_k^*\otimes 1+e_k^*e_k\otimes u_k), k\in M_x.$$
Set $V=V_{i_1}\ldots V_{i_{|M_x|}}$, where $M_x=\{i_1,\ldots, i_{|M_x|}\}$ and $i_k< i_l$ if $k<l$. 
Then 
\begin{equation*}
    \begin{split}
       & V\rho_x(a_i)V^* = \rho_x(a_i)=\prod_{k\in L_x} (e_k e_k^* + e^{\pi i \Theta_{i,k}} e_k^* e_k) e_i \otimes 1, i \in L_x, \\
       & V\rho_x(a_i)V^* = \prod_{k\in L_x} (e_k e_k^* + e^{\pi i \Theta_{i,k}} e_k^* e_k)\otimes \\&\qquad\otimes \left(\left(\prod_{k\in M_x, k<i} ( e_k^* e_k + e^{2\pi i \Theta_{i,k}} e_k e_k^*)\otimes 1_H\right)\times\right. \\&\times\left.\left(\sqrt{x_i}\prod_{k\in M_x,k\geq i}(e_k^* e_k + e^{4\pi i \Theta_{i,k}} e_k e_k^*)e_i\otimes u_i^2 + \sqrt{1 - x_i} e_i^*\otimes 1_H\right)\right) ,  i \in M_x, \\
       & V\rho_x(a_i)V^* = \prod_{k \in L_x} (e_k e_k^* + e^{\pi i \Theta_{i,k}} e_k^* e_k) \otimes\\&\qquad\otimes \prod_{k\in M_x}(e_k^* e_k + e^{2\pi i \Theta_{i,k}} e_k e_k^*)\otimes\frac{1_H}{\sqrt{2}} _i \ ,  i \in R_x.
    \end{split}
\end{equation*}
It is easy to see that the family $\{u_i^2:i\in M_x\}\cup\{u_i:i\in R_x\}$ forms a representation of $C(\mathbb T^{M_x\sqcup R_x}_{\Sigma})$. Moreover, any such family with $v_i$ instead of $u_i^2$, $i\in M_x$, and $v_i$ instead of $u_i$, $i\in R_x$, where $(v_i)_{i\in M_x\sqcup R_x}$ defines a representation of $C(\mathbb T^{M_x\sqcup R_x}_{\Sigma})$, is a representation of $\mathsf{CAR}_{\Theta}$.

Fix $x\in \left[0,\frac{1}{2}\right]^{M_x}$ and let $C$ be an operator intertwining the representations corresponding to families $\mathbb V=(v_i)_{i\in M_x\sqcup R_x}$ and $\mathbb W=(w_i)_{i\in M_x\sqcup R_x}$ acting on $H_{\mathbb V}$ and $H_{\mathbb W}$ respectively.  Denote them 
by $\tau_{\mathbb V}$ and $\tau_{\mathbb W}$ respectively; we have $C\tau_{\mathbb V}(a)=\tau_{\mathbb W}(a)C$, $a\in  \mathsf{CAR}_{\Theta}$, i.e. $C\in \text{Hom}(\tau_{\mathbb V},\tau_{\mathbb W})$.  In particular,  
\begin{equation}\label{eq_inter}
C\tau_{\mathbb V}(a_i^*a_i)=\tau_{\mathbb W}(a_i^*a_i)C,
i\in L_x\sqcup M_x\sqcup R_x. 
\end{equation}
We have
\begin{equation}\label{aiai}
\tau_{\mathbb V}(a_i^*a_i)=\left\{\begin{array}{ll}e_i^*e_i\otimes 1\otimes 1_{H_{\mathbb V}},&i\in L_x,\\
1\otimes ((1-x_i)e_ie_i^*+x_ie_i^*e_i)\otimes 1_{H_{\mathbb V}},&i\in M_x,\\
1\otimes 1\otimes 1_{H_{\mathbb V}},& i\in R_x.\end{array}\right.
\end{equation}
Therefore, it is easy to see that (\ref{eq_inter}) implies that $C=\sum_i p_i\otimes C_i$, where $p_i=\prod_{k\in L_x\sqcup M_x}q_k^i$ with $q_k^i\in \{e_ke_k^*, e_k^*e_k\}_{k\in L_x\sqcup M_x}\subset \otimes_{k\in L_x\sqcup M_x}M_2$ ($=(\otimes_{k\in L_x} M_2)$ $\bigotimes(\otimes_{k\in M_x} M_2)$), and $C_i\in \mathcal B(H_{\mathbb V}, H_{\mathbb W})$; the summation is over all possible products $ p_i=\prod_{k\in L_x\sqcup M_x}q_k^i$.

The condition $C\tau_{\mathbb V}(a_k)=\tau_{\mathbb W}(a_k)C$ for $k\in L_x$ is equivalent to $$\sum _i\alpha_{i,k}p_ie_k\otimes C_i=\sum _i\alpha_{i,k}e_kp_i\otimes C_i$$ for some non-zero $\alpha_{i,k}$. As 
$$p_ie_k=\left\{\begin{array}{ll} 0,& \text{if } q_k^i=e_k^*e_k\\
e_kp_{\sigma_k(i)}, &\text{if } q_k^i=e_ke_k^*\end{array}\right.$$
here $q_k^{\sigma_k(i)}=e_k^*e_k$  if $q_k^i=e_ke_k^*$ and vice versa and $q_j^{\sigma_k(i)}=q_j^i$ otherwise (i.e. we swap the projection $q_k^i$ to the other possible value for the $k$th factor), we obtain $C_i=C_{\sigma_k(i)}$ for all $k\in L_x$. 
Similarly, the condition $C\tau_{\mathbb V}(a_k)=\tau_{\mathbb W}(a_k)C$ for $k\in M_x$  implies first that $C_i=C_{\sigma_k(i)}$, $k\in M_x$, giving now that all $C_i$'s are equal; call the common value $C'$ and get $C=1\otimes C'$.   Then we obtain  that $C'v_k=w_kC'$ for all $k\in M_x$. The condition $C\tau_{\mathbb V}(a_k)=\tau_{\mathbb W}(a_k)C$ for $k\in R_x$ gives $C'v_k=w_kC'$.  Therefore we have a bijection $\text{Hom}(\tau_{\mathbb V},\tau_{\mathbb W})\to\text{Hom}(\mathbb V,\mathbb W)$, $1\otimes C\mapsto C$. From this it easily follows that $\tau_{\mathbb V}$ is irreducible iff $\mathbb V$ defines an irreducible representation  $C(\mathbb T^{M_x\sqcup R_x}_{\Sigma})$. 



That $\tau_x$ and $\tau_y$ are not unitary equivalent for $x\ne y$ follows from the fact that the spectrum of $\tau_x(a_i^*a_i)$ is in $\{1, x_i, 1-x_i\}$ if $i\in M_x$, see (\ref{aiai}).
 \end{proof}

\section{Classification of $\mathsf{CAR}_{\Theta}$}
This section contains the main result of the paper and concerns the classification of $ \mathsf{CAR}_{\Theta}$ up to isomorphism. To obtain the result we will employ another $C(K)$-structure coming from the center of  $\mathsf{CAR}_{\Theta}$ and relate it to the $C\left(\left[0,\frac{1}{2}\right]\right)$-structure on  the algebra. We will then use $K$-theoretical arguments applied to the fibers to derive the result.

Let $\Theta_1$ and $\Theta_2$ be skew-symmetric real $n\times n$ matrices. Suppose $\varphi : \mathsf{CAR}_{\Theta_1} \rightarrow \mathsf{CAR}_{\Theta_2}$ is an isomorphism. It induces an isomorphism of the centers and a homeomorphism $\alpha: \text{spec }Z(\mathsf{CAR}_{\Theta_2}) \rightarrow \text{spec }Z(\mathsf{CAR}_{\Theta_1})$ of their Gelfand spectrum. Let $Z_{\Theta_i}=\text{spec }Z(\mathsf{CAR}_{\Theta_i})$, $i=1,2$.
 We have a natural $C(Z_{\Theta_i})$-structure  on $ \mathsf{CAR}_{\Theta_i}$ given by the inverse of the Gelfand transform $\hat g\mapsto g$, $g\in  Z(\mathsf{CAR}_{\Theta_i})$, $i=1,2$: $\Phi_i(\hat g)\cdot a=ga$, $a\in \mathsf{CAR}_{\Theta_i}$. Letting  $$I_z^\Theta=\{ga: a\in  \mathsf{CAR}_{\Theta}, \hat g(z)=0\},  z\in Z_{\Theta}, $$
we have  the following commutative diagram:
\[
\begin{tikzcd}
0 \arrow[r] &  I_{\alpha(z)}^{\Theta_1} \arrow[r] \arrow[d, "\varphi \vert_{I_{\alpha(z)}^{\Theta_1}}"]& \mathsf{CAR}_{\Theta_1} \arrow[r] \arrow[d,"\varphi"] & \mathsf{CAR}_{\Theta_1}(\alpha(z)) \arrow[r] \arrow[d, dashed] & 0 \\
0 \arrow[r] &  I_{z}^{\Theta_2} \arrow[r] & \mathsf{CAR}_{\Theta_2} \arrow[r] & \mathsf{CAR}_{\Theta_2}(z) \arrow[r] & 0
\end{tikzcd}
\]
which gives the isomorphisms $$\mathsf{CAR}_{\Theta_1}({\alpha(z)}) \simeq \mathsf{CAR}_{\Theta_2}(z)$$ for every $z\in Z_{\Theta_2}$. 

The $C\left(K_n\right)$-structure on  $\mathsf{CAR}_{\Theta}$ induces an injective homomorphism from $C\left(K_n\right)$ to $C(Z_{\Theta})$ and hence a canonical continuous surjection 
$\pi : Z_{\Theta} \rightarrow K_n$. We also have for all $z\in Z_{\Theta}$  that $I_{\pi(z)}$ is an ideal in $I_z^\Theta$ and hence 
$$\mathsf{CAR}_{\Theta}(z)\simeq \mathsf{CAR}_{\Theta}/I_z^\Theta\simeq (\mathsf{CAR}_{\Theta}/I_{\pi(z)})/(I_z^\Theta/I_{\pi(z)})$$ so that  $\mathsf{CAR}_{\Theta}(z)$ is a quotient of $\mathsf{CAR}_{\Theta}(\pi(z))$.

\begin{definition}
Recall $L_x$, $M_x$ and $R_x$, $x\in K_n$, from Section 7 and for each $z\in  Z_{\Theta}$ we define the face signature to be  $\mathsf{face}(z)=(|L_{\pi(z)}|,|M_{\pi(z)}|, |R_{\pi(z)}|)$.
\end{definition}

\begin{definition}
We say that a real skew-symmetric $n\times n$ matrix $\Theta$ is {\it irrational} if whenever $p\in\mathbb Z^n$ satisfies $e^{2\pi i\langle p,\Theta (q)\rangle}=1$ for all $q\in \mathbb Z^n$ then $p=0$. 
\end{definition}
We note that some authors choose to call such $\Theta$ non-degenerate, see e.g. \cite{phillips}. 

\smallskip

We now give a description of the fibers of $\mathsf{CAR}_\Theta$ over $Z_\Theta$ using the above connection with $C\left(K_n \right)$-structure and the description of fibers given in Proposition \ref{isomorphism}.

Let $\Theta$ be an irrational skew-symmetric $n\times n$-matrix. For $z\in Z_{\Theta}$ set 
$x = \pi(z) \in K_n$ and   $l = |L_x|, m = |M_x|, r = |R_x|$. The description splits in the following four cases:
\begin{enumerate}
    \item If $m + r \geq 2$ then $\mathsf{CAR}_{\Theta}(x) \simeq Cl_{2l} \otimes C(\mathbb{T}^{m+r}_{\Theta_{M_x \sqcup R_x}}) \rtimes _{\beta_{\Theta}^x}\mathbb{Z}_2^m$. Since 
    $\Theta_{M_x \sqcup R_x}$ is irrational, $Z(\mathsf{CAR}_{\Theta}(x)) \simeq \mathbb C$.  From this one can easily derive that $I_{\pi(z)}=I_z^\Theta$ and hence   
    $$\mathsf{CAR}_{\Theta}(z)\simeq Cl_{2l} \otimes C(\mathbb{T}^{m+r}_{\Theta_{M_x \sqcup R_x}}) \rtimes_{\beta_{\Theta}^x} \mathbb{Z}_2^m.$$
    \item If $l = n - 1, \ m = 1$ then $\mathsf{CAR}_{\Theta_2}(z)$ is a quotient of $(\mathsf{CAR}_{\Theta_2})(x) \simeq Cl_{2n - 2} \otimes C(\mathbb{T}) \rtimes _{\beta_{\Theta}^x}\mathbb{Z}_2$. As $C(\mathbb{T}) \rtimes _{\beta}\mathbb{Z}_2 \simeq M_2(C(\mathbb{T}))$ (see e.g. \cite[Proposition 3.4]{choi-latre}),
    $\mathsf{CAR}_{\Theta_2}(x) \simeq Cl_{2n} \otimes C(\mathbb{T})$ with  all quotients being of the form $Cl_{2n} \otimes C(K)$ for some closed subset  $K \subset \mathbb{T}$.
    \item If $l = n - 1, \ r = 1$ then  $\mathsf{CAR}_{\Theta}(z)$ is a quotient of $\mathsf{CAR}_{\Theta}(x) \simeq Cl_{2n - 2} \otimes C(\mathbb{T})$.  All such quotients have the form $Cl_{2n-2} \otimes C(K)$ for a closed subset $K \subset \mathbb{T}$.
    \item If $l = n$ then $\mathsf{CAR}_{\Theta}(x) \simeq Cl_{2n}\simeq \mathsf{CAR}_{\Theta}(z)$.
\end{enumerate}
To prove the main result we need the following auxiliary lemmas. 

\smallskip{}

\begin{lemma}\label{torus} Let $\Theta$ be irrational and let $\sigma\in \text{Aut}(C(\mathbb T^n_\Theta))$, given by $\sigma(u_1)=-u_1$, $\sigma(u_k)=u_k$, $k>1$. Then
$$C(\mathbb T^n_\Theta)^\sigma\simeq C(\mathbb T^n_{\Theta^{(1)}}),$$
where $\Theta^{(1)}_{i,j}=2\Theta_{i,j}$ if either $i$ or $j=1$ and $\Theta^{(1)}_{i,j}=\Theta_{i,j}$ otherwise. 
\end{lemma}

\begin{proof}
We note first that $C(\mathbb T^n_\Theta)^\sigma=\{x+\sigma(x): x\in C(\mathbb T^n_\Theta)\}$ from which it is easy to see using approximation arguments that $C(\mathbb T^n_\Theta)^\sigma$ equals  the $C^*$-subalgebra $C^*(u_1^2,u_2,\ldots, u_n)$, generated by $u_1^2$, $u_2,\ldots, u_n$. Furthermore, the map $u_1\mapsto u_1^2$, $u_k\mapsto u_k$, $k>1$, extends to a surjective $*$-homomorphism from $C(\mathbb T_{\Theta^{(1)}}^n)$ to $C^*(u_1^2,u_2,\ldots, u_n)$. The statement now follows from the  simplicity of $C(\mathbb T_{\Theta^{(1)}}^n)$, see e.g. \cite[Theorem 1.9]{phillips}.
\end{proof}
For a skew-symmetric real matrix $\Theta$ of size $n=m+r$ let $\Sigma$ be given by
$$\Sigma_{i,j}=\left\{\begin{array}{cl}4\Theta_{i,j}& i,j\leq m\\2\Theta_{i,j}& \text{ either } i\leq m \text{ or } j\leq m,\\
\Theta_{i,j}& i,j>m\end{array}\right.$$
Define $\beta_\Theta:\mathbb Z^m_2\to\text{Aut}(C(\mathbb T^{m+r}_\Theta))$ by $\beta_\Theta(\omega)(u_k)=\omega_ku_k$ for $k\leq m$ and $\beta_\Theta(\omega)(u_k)=u_k$ if $k>m$, where $\omega=(\omega_1,\ldots,\omega_m)$.

\begin{lemma}\label{cross_lemma}
Let $\Theta$, $\Sigma$ and $\beta_\Theta$ be as above. Then 
\[ C(\mathbb{T}^{m+r}_\Theta) \rtimes_{\beta_\Theta} \mathbb{Z}_2^m \simeq Cl_{2m}\otimes C(\mathbb{T}^{m+r}_{\Sigma}). \]
\end{lemma}
\begin{proof}
Let first $m=1$ and write $\sigma$ for $\beta_\Theta$.
The arguments as in Theorem \ref{theorem65} show that $$C(\mathbb T^{1+r}_\Theta)\rtimes_{\sigma} \mathbb{Z}_2\simeq (M_2(C(\mathbb T^{1+r}_\Theta)))^{\tilde\sigma},$$ where $\tilde\sigma=\text{Ad}W\otimes\sigma$ and $W=\left(\begin{array}{cc} 1&0\\0& -1\end{array}\right)$.
Furthermore, if $U=\left(\begin{array}{cc}0&1\\u_1&0\end{array}\right)$, then $UM_2(C(\mathbb T^{1+r}_\Theta))^{\tilde\sigma}U^*=M_2(C(\mathbb T^{1+r}_\Theta)^{\sigma})$,
as
\begin{equation}\nonumber
M_2(C(\mathbb T^{1+r}_\Theta))^{\tilde\sigma}=\left\{ \left[ \begin{array}{cc}
    A & B \\
    C & D
\end{array} \right] : A, D \in C(\mathbb T^{1+r}_\Theta)^{\sigma}, B, C\in C(\mathbb T^{1+r}_\Theta)^{\sigma}(-1) \right\} \end{equation}
 where $\cl A^\sigma(-1)=\{ a\in \cl A: \sigma(a)=-a\}$. 
 This together with Lemma \ref{torus} yields the statement for $m=1$. 
 To see it for general $m$ we note first that
 $$C(\mathbb T^{m+r}_\Theta)\rtimes_{\beta_\Theta}\mathbb Z_2^m\simeq(C(\mathbb T^{1+(m-1+r)}_\Theta)\rtimes_{\sigma}\mathbb Z_2)\rtimes_{\beta_\Theta'}\mathbb Z_2^{m-1}
 $$
 which together with the previous result and simple calculations gives
 $$C(\mathbb T^{m+r}_\Theta)\rtimes_{\beta_\Theta}\mathbb Z_2^m\simeq Cl_2\otimes C(\mathbb T^{m+r}_{\Theta^{(1)}})\rtimes_{\beta_\Theta^{(1)}}\mathbb Z_2^{m-1},$$
 where $\Theta^{(1)}$ is as in Lemma \ref{torus},
  $\beta_\Theta'$ acts as $\beta_\Theta$ on $C(\mathbb{T}_\Theta^{1 + (m - 1 + r)})$ and identically on the generator of $\mathbb{Z}_2$, and $\beta_\Theta^{(1)}:\mathbb Z_2^{m-1}\to\text{Aut}(C(\mathbb T^{m+r}_{\Theta^{(1)}})$ is given by $\beta_\Theta^{(1)}(\omega)(u_i)=\omega_iu_i$ if $2\leq i\leq m$ and $\beta_\Theta^{(1)}(\omega)(u_1)=u_1$ for  $\omega=(\omega_2,\ldots,\omega_m)$. 
 The statement now follows by the successive application of the above argument.
\end{proof}

\begin{lemma}\label{k_0}
For $z\in Z_{\Theta}$ set 
 $m = |M_{\pi(z)}|$ and  $r = |R_{\pi(z)}|$. If $m + r > 1$ and $\Theta$ is irrational then
\[ K_0(\mathsf{CAR}_{\Theta}(z)) \simeq \mathbb{Z}^{2^{m + r - 1}}. \]
\end{lemma}
\begin{proof}
If $m + r > 1$ then $\mathsf{CAR}_{\Theta}(z) \simeq Cl_{2l} \otimes C(\mathbb{T}^{m+r}_{\Theta_{M_x \sqcup R_x}}) \rtimes _{\beta_{\Theta}^x}\mathbb{Z}_2^m$ and by Lemma \ref{cross_lemma} 
$$\mathsf{CAR}_{\Theta}(z)\simeq Cl_{2l+2m}\otimes C(\mathbb T^{m+r}_\Sigma).$$ 
Thus, by Proposition \ref{Rieff_K_theory} 
$$K_0(\mathsf{CAR}_{\Theta}(z)) \simeq K_0(C(\mathbb{T}^{m+r}_{\Sigma})) \simeq K_0(C(\mathbb{T}^{m+r})) \simeq \mathbb{Z}^{2^{m + r - 1}}.$$ 
\end{proof}

\begin{lemma}
Let $\theta_1, \theta_2, \theta_3\in \mathbb R\setminus \mathbb Q$. The $C^*$-algebras
\[Cl_{2n-4} \otimes C(\mathbb{T}^2_{\theta_1}), \ Cl_{2n-4} \otimes C(\mathbb{T}^2_{\theta_2}) \rtimes_{\beta_1} \mathbb{Z}_2, \ Cl_{2n-4} \otimes C(\mathbb{T}^2_{\theta_3}) \rtimes_{(\beta_1\times\beta_2)} \mathbb{Z}_2^2 \]
are mutually non-isomorphic.
\end{lemma}
\begin{proof}
It is known that $C(\mathbb{T}^2_\theta)$ and $\otimes_{k\in S}M_{n(k)}\otimes C(\mathbb{T}^2_\theta)$  are  $C^*$-algebras with unique normalized trace which we denote by $\tr$. By a result of Rieffel (\cite[Theorem 1.2, Proposition 1.3]{rieff}), $\tr(\mathcal P(M_n\otimes C(\mathbb{T}^2_\theta))) = n^{-1}(\mathbb{Z} + \theta \mathbb{Z}) \cap [0,1]$, where $\mathcal P(A)$  is the set of projections of $A$. 
Therefore 
\[ \tr(\mathcal P(Cl_{2n-4} \otimes C(\mathbb{T}^2_{\theta_1}))) = \frac{1}{2^{n-2}}\tr(\mathcal P(C(\mathbb{T}^2_{\theta_1}))) = \frac{1}{2^{n-2}}(\mathbb{Z} + \theta_1 \mathbb{Z}) \cap [0,1], \] 
\[\begin{split} \tr(\mathcal P(Cl_{2n-4} \otimes C(\mathbb{T}^2_{\theta_2}) \rtimes_{\beta_1} \mathbb{Z}_2)) &\stackrel{\text{Lemma }\ref{cross_lemma}}{=}\tr(\mathcal P(Cl_{2n-2} \otimes C(\mathbb{T}^2_{2{\theta_2}}))\\&=   \frac{1}{2^{n-1}}(\mathbb{Z} + 2\theta_2 \mathbb{Z}) \cap [0,1]. \end{split}\]
and 
\[ \begin{split}\tr(\mathcal P(Cl_{2n-4} \otimes C(\mathbb{T}^2_{\theta_3}) \rtimes_{(\beta_1\times\beta_2)} \mathbb{Z}_2^2)) &\stackrel{\text{Lemma }\ref{cross_lemma}}{=}\tr(\mathcal P(Cl_{2n} \otimes C(\mathbb{T}^2_{4{\theta_3}})))\\&=   \frac{1}{2^{n}}(\mathbb{Z} + 4\theta_3 \mathbb{Z}) \cap [0,1]. \end{split}\] 
showing that the $C^*$-algebras $Cl_{2n-4} \otimes C(\mathbb{T}^2_{\theta_1})$, $Cl_{2n-4} \otimes (C(\mathbb{T}^2_{\theta_2}) \rtimes_{\beta_1} \mathbb{Z}_2)$ and $Cl_{2n-4} \otimes (C(\mathbb{T}^2_{\theta_3}) \rtimes_{(\beta_1\times\beta_2)} \mathbb{Z}_2^2)$ are mutually non-isomorphic. 
\end{proof}

\begin{lemma}\label{lemma81}
Let $\Theta_1$ and $\Theta_2$ be irrational skew-symmetric $n\times n$ matrices and 
let $\varphi : \mathsf{CAR}_{\Theta_1} \rightarrow \mathsf{CAR}_{\Theta_2}$ be an isomorphism with the induced homeomorphism $\alpha: Z_{\Theta_2}\to Z_{\Theta_1}$. If   $z \in Z_{\Theta_2}$ satisfies $\mathsf{face}(z) = (n - 2, 0, 2)$ then $\mathsf{face}(z) = \mathsf{face}(\alpha(z))$. 
\end{lemma}
\begin{proof}

We observe first that if $z\in Z_\Theta$ is such that  $m=|M_{\pi(z)}|$ and $r=|R_{\pi(z)}|$ satisfy 
$m + r \leq 1$, then $z$ is either of type $(2)$, $(3)$ or $(4)$ and hence  $\mathsf{CAR}_{\Theta}(z)$ is a $C^*$-algebra of the form $M_n(C(X))$, which is either finite-dimensional or non-simple, while if  $m + r > 1$ then  $\mathsf{CAR}_{\Theta}(z)$ is infinite-dimensional and simple.   
From this we can conclude that if $\Theta_1$ and $\Theta_2$ are irrational then 
$|M_{\pi(\alpha(z))}|+|R_{\pi(\alpha(z))}|\leq 1$ whenever $|M_{\pi(z)}|+|R_{\pi(z)}|\leq 1$. 
Therefore if $\mathsf{face}(z) = (n - 2, 0, 2)$ then $|M_{\pi(\alpha(z))}|+|R_{\pi(\alpha(z))}|$ is necessarily larger than 1, and by Lemma \ref{k_0} must be exactly $2$. 
This gives that the possible values of $\mathsf{face}(\alpha(z))$ are $(n-2,0,2)$, $(n-2,1,1)$ and $(n-2,2,0)$. Hence, as
$\mathsf{CAR}_{\Theta_1}(\alpha(z))\simeq \mathsf{CAR}_{\Theta_2}(z)$, to prove the statement it is enough to see that $\mathsf{CAR}_{\Theta}(z)$ are non-isomorphic for different $z$ with $(m,r)\in \{(0, 2), (1, 1), (2, 0)\}$. But for $(m, r) = (0, 2), (1, 1), (2, 0)$, $\mathsf{CAR}_{\Theta}(z)$ is isomorphic to  $Cl_{2n-4} \otimes C(\mathbb{T}^2_\theta)$, $Cl_{2n-4} \otimes C(\mathbb{T}^2_\theta) \rtimes_{\beta_1} \mathbb{Z}_2$ and $Cl_{2n-4} \otimes C(\mathbb{T}^2_\theta) \rtimes_{(\beta_1\times\beta_2)} \mathbb{Z}_2^2$ respectively. Thus Lemma 8.5 concludes the proof.


\end{proof}

A matrix $P=(p_{i,j})_{i,j=1}^n\in M_n$ is called a signed permutation matrix if there exists $(\sigma, b)\in S_n\times\{0,1\}^n$ such that $p_{i,j}=(-1)^{b_i}\delta_{j,\sigma(i)}$.
We are now ready to prove our main results.

\begin{theorem}\label{clas_tori}
Let $\Theta_1$ and $\Theta_2$ be irrational $n\times n$-matrices. 
\begin{enumerate}
    \item If $P$ is a signed permutation matrix then $\Theta_1 = P \Theta_2 P^t$ implies $\mathsf{CAR}_{\Theta_1} \simeq \mathsf{CAR}_{\Theta_2}$.
    \item If $\mathsf{CAR}_{\Theta_1} \simeq \mathsf{CAR}_{\Theta_2}$ then $(\Theta_2)_{i,j} = \pm (\Theta_1)_{\sigma(i,j)} \mod \mathbb{Z}$ for a bijection $\sigma$ of the set $\{(i,j) : i < j, \ i, j = 1, \ldots, n\}$.
\end{enumerate}

\end{theorem}

\begin{proof}
(1) If $P$ is a signed permutation matrix which corresponds to a signed permutation $(\sigma, b) \in S_n \times \{1, *\}^n$ then the corresponding isomorphism is given by
\[ \psi_P(a_i) = a_{\sigma(i)}^{b_i}. \]

(2) Let $z$ be the unique element of $Z_{\Theta_2}$ such that $\pi(z) = \frac{1}{2}(\delta_i + \delta_j)$, $i< j$. Since $\mathsf{face}(z) = (n - 2, 0, 2)$, by Lemma \ref{lemma81} $\mathsf{face}(\alpha(z)) = (n - 2, 0 , 2)$ and hence  $\alpha(z) = \frac{1}{2}(\delta_k + \delta_l)$,  where $(k, l) = \sigma(i, j)$ for  a bijection $\sigma$ of the set $\{(i,j) : i < j, \ i, j = 1, \ldots, n\}$. Thus 
\begin{eqnarray*}
\mathsf{CAR}_{\Theta_2}(z)& \simeq& Cl_{2n-4} \otimes C(\mathbb{T}^2_{(\Theta_2)_{i,j}})\simeq \mathsf{CAR}_{\Theta_1}(\alpha(z))\\& \simeq& Cl_{2n-4} \otimes C(\mathbb{T}^2_{(\Theta_1)_{\sigma(i,j)}}).
\end{eqnarray*}
 and by \cite[Theorem 3]{rieff} $(\Theta_2)_{i,j} = \pm (\Theta_1)_{\sigma(i, j)} \mod \mathbb{Z}$. 
\end{proof}

For $\theta\in \mathbb R$ write simply $\mathsf{CAR}_{\theta}$ for $\mathsf{CAR}_{\Theta}$ if $n=2$ and $\Theta_{1,2}=\theta$. In this case we have the full classification similar to the classification of two-dimensional non-commutative tori. 

\begin{corollary}\label{clas_cor}
If  $\theta_1$, $ \theta_2$ are irrational numbers then $\mathsf{CAR}_{\theta_1} \simeq \mathsf{CAR}_{\theta_2}$ iff $\theta_1 = \pm \theta_2 \mod \mathbb{Z}$.
\end{corollary}

\medskip 

{\bf Acknowledgement. } It is a pleasure to express our gratitude to Daniil Proskurin and Magnus Goffeng for helpful discussions. We would also like to thank the anonymous referee for
valuable suggestions that led to improvements in the paper.

 \end{document}